\documentclass[12pt,letterpaper, final]{amsart}
\usepackage{amssymb, amsmath, amsthm}
\usepackage{mathrsfs}  
\usepackage{graphicx}
\usepackage[colorlinks=true, citecolor=blue, linkcolor=red, urlcolor=blue]{hyperref}
\usepackage[text={6in,8in},centering]{geometry}
 \usepackage{xcolor}

\usepackage{mathrsfs}  

\usepackage{bm}
\usepackage{ifdraft}
\ifoptionfinal{
\usepackage[disable]{todonotes}
}{
\usepackage[norefs, nocites]{refcheck}

\usepackage[notref, notcite]{showkeys}
\usepackage[bordercolor=white, color=white]{todonotes}
}

\makeatletter\providecommand\@dotsep{5}\def\listtodoname{List of Todos}\def\listoftodos{\hypersetup{linkcolor=black}\@starttoc{tdo}\listtodoname\hypersetup{linkcolor=blue}}\makeatother
\usepackage{hyperref}
\usepackage{etoolbox}

\newtheorem{theorem}{Theorem}[section]
\newtheorem{lemma}[theorem]{Lemma}
\newtheorem{corollary}[theorem]{Corollary}
\newtheorem{proposition}[theorem]{Proposition}
\theoremstyle{definition}
\newtheorem{definition}[theorem]{Definition}

\theoremstyle{remark}
\newtheorem{remark}[theorem]{Remark}

\numberwithin{equation}{section}

\def\C{\mathbb C}
\def\R{\mathbb R}

\def\N{\mathbb N}

\renewcommand{\leq}{\leqslant}
\renewcommand{\geq}{\geqslant}

\def\p{\partial}
\DeclareMathOperator{\supp}{supp}

\newcommand*\xbar[1]{%
   \hbox{%
     \vbox{%
       \hrule height 0.5pt % The actual bar
       \kern0.5ex%         % Distance between bar and symbol
       \hbox{%
         \ensuremath{#1}%
       }%
     }%
   }%
} 

\title[]{Calder\'{o}n problem for fractional Schr\"{o}dinger operators on closed Riemannian manifolds}
\author[Feizmohammadi]{Ali Feizmohammadi}
\address{A. Feizmohammadi, Department of Mathematics, University of Toronto, Road Deerfield Hall, 3008K Mississauga, ON L5L 1C6}
\email{ali.feizmohammadi@utoronto.ca}

\author[Krupchyk]{Katya Krupchyk}
\address
        {K. Krupchyk, Department of Mathematics\\
University of California, Irvine\\
CA 92697-3875, USA }

\email{katya.krupchyk@uci.edu}

\author[Uhlmann]{Gunther Uhlmann}
\address
       {G. Uhlmann, Department of Mathematics\\
       University of Washington\\
       Seattle, WA  98195-4350\\
       USA\\
        and Institute for Advanced Study of the Hong Kong University of Science and Technology}
\email{gunther@math.washington.edu}

\begin{document}

%\begin{titlepage}
%\maketitle
%\end{titlepage}

\maketitle
\begin{abstract}
We study an analog of the anisotropic Calder\'on problem for fractional Schr\"odinger operators $(-\Delta_g)^\alpha + V$ with $\alpha \in (0,1)$ on closed Riemannian manifolds of dimensions two and higher. We prove that the knowledge of a Cauchy data set of solutions of the fractional Schr\"odinger equation, given on an open nonempty a priori known subset of the manifold determines both the Riemannian manifold up to an isometry and the potential up to the corresponding gauge transformation, under certain geometric assumptions on the manifold as well as the observation set. Our method of proof is based on: (i) studying a new variant of the Gel'fand inverse spectral problem without the normalization assumption on the energy of eigenfunctions, and (ii) the discovery of an {\em entanglement} principle for nonlocal equations involving two or more compactly supported functions. Our solution to (i) makes connections to antipodal sets as well as local control for eigenfunctions and quantum chaos, while (ii) requires sharp interpolation results for holomorphic functions. We believe that both of these results can find applications in other areas of inverse problems.
\end{abstract}

\section{Introduction}
Let $(M,g)$ be a smooth closed and connected Riemannian manifold of dimension $n \geq 2$. By closed, we mean that the manifold is compact and without boundary. Let $V \in C^{\infty}(M)$ and let $O\subset M$ be a nonempty open set such that $M\setminus\overline{O}$ is also nonempty. We will refer to $O$ as the \emph{observation set}. We will assume that the observation set $O$, together with $g|_{O}$ and $V|_{O}$ are known but that the complement of this set is inaccessible to us, that is to say, the differential and topological structure of $M\setminus O$ together with restrictions of the metric $g$ and the function $V$ on this set are all a priori unknown. 

Before stating the main inverse problem of interest, we recall the well known and widely open anisotropic Calder\'{o}n problem that can be equivalently stated in the setting of closed manifolds as follows. Suppose that we are given the knowledge of the Cauchy data set
$$ {\widetilde {\mathcal C}}^O_{M,g,V} = \{(u|_{O}, (-\Delta_g u)|_{O})\,:\, \text{$u \in C^{\infty}(M),$ $-\Delta_gu+Vu=0$ on $M\setminus \overline O$}\},$$
associated to solutions to Schr\"{o}dinger equations. Here, $-\Delta_g$ is the positive Laplace--Beltrami operator on $M$, that is given in local coordinates via 
\[
 -\Delta_g = -\frac{1}{\sqrt{|g|}} \sum_{j,k=1}^n \frac{\p}{\p x^j}\left( \sqrt{|g|} \,g^{jk} \frac{\p }{\p x^k}  \right),
 \]
where $(g^{jk})=(g_{jk})^{-1}$ and $|g|=\det(g_{jk})$. The anisotropic Calder\'on problem for the Schr\"odinger equation is the question of whether the Cauchy data set $\widetilde{ {\mathcal C}}^O_{M,g,V}$ determines the isometry class of the manifold $(M,g)$, as well as the function $V$, up to the corresponding gauge transformation. This is one of the most well known and widely open problems in the literature of inverse problems and goes back to its introduction in the pioneering  paper of Calder\'{o}n in \cite{Ca}. When $n = 2$, the problem is solved, see \cite{LaUh, Nachman, Bukhgeim_2008, Imanuvilov_Uhlmann_Yamamoto_2011, Imanuvilov_Uhlmann_Yamamoto_2012, Guillarmou_Tzou_2011, Carstea_Liimatainen_Tzou_2024}, under an additional gauge arising from conformal invariance of the Laplace--Beltrami operator. In dimensions three and higher, the problem is solved for real-analytic manifolds (see, e.g., \cite{KV,LU, LaUh}), but it remains widely open for the case of smooth manifolds. We refer the reader to \cite{DKSU, DKLS} for state of the art results and to \cite{Uh} for a survey on the anisotropic Calder\'{o}n problem.

The main goal of this manuscript is to obtain uniqueness results for an analog of the anisotropic Calder\'on problem associated to fractional Schr\"odinger equations. To state the problem, recall that the Laplace--Beltrami operator $-\Delta_g$ is a self-adjoint operator on $L^2(M)$ with the domain $\mathcal{D}(-\Delta_g) = H^2(M)$, the standard Sobolev space on $M$. Letting $\alpha \in (0,1)$ and using the functional calculus for self-adjoint operators, we define the fractional Laplacian $(-\Delta_g)^\alpha$ as an unbounded self-adjoint operator on $L^2(M)$ with the domain $\mathcal{D}((-\Delta_g)^\alpha) = H^{2\alpha}(M)$ (see Section \ref{sec_pre} for further details).

Let $V \in C^\infty(M)$ and consider the fractional Schr\"odinger operator
\begin{equation}
\label{Schrodinger_eq}
P_{M,g,V} = (-\Delta_g)^\alpha + V.
\end{equation}
Associated to \eqref{Schrodinger_eq}, we define the Cauchy data set on $O$ as follows:
\begin{equation}
\label{DN_map}
\mathcal{C}_{M,g,V}^O = \{(u|_O, ((-\Delta_g)^\alpha u)|_O) \,:\, \text{$u \in C^\infty(M)$, $P_{M,g,V} \, u = 0$ on $M \setminus \overline{O}$} \}.
\end{equation}

We study the following inverse problem, which may be viewed as a variant of the anisotropic Calder\'on problem stated for the nonlocal Schr\"{o}dinger operator \eqref{Schrodinger_eq}, as follows:
\begin{itemize}
\item[{\bf (IP)}] Does the knowledge of the Cauchy data set $\mathcal{C}_{M,g,V}^O$ uniquely determine the differential and topological structure of $M$, the metric $g$, and the function $V$?
\end{itemize}

When $V$ is identically zero on $M$ and $(M,g)$ is sought after, (IP) has been solved in complete geometric generality in \cite{FGKU},  see also the precursor \cite{Fei21}, which solved the problem under a Gevrey analyticity assumption on the set $O$, and the follow-up work \cite{CO}, which deals with noncompact manifolds. For an alternative perspective on the approach of \cite{FGKU}, see \cite{Ruland_2023}. In the case of Euclidean space equipped with a known Riemannian metric, the uniqueness for the recovery of the potential function $V$ is established in \cite{Ghosh_Lin_Xiao_2017}. We also mention the recent works \cite{Zimmermann_2023, Lin_Zimmermann_2024}, where in the Euclidean space, the inverse problems of recovering a scalar leading term and a zeroth order potential in the nonlocal diffuse optical tomography equations were studied.

To the best of our knowledge, the inverse problem of recovering both the Riemannian metric and the potential, and more generally, the differential and topological structure of the manifold, has remained open up to now. One of the main challenges here is that both a nonlocal term, namely $(-\Delta_g)^\alpha u$, and a local term, $Vu$, are simultaneously present in the same equation. Modifications of the arguments in \cite{FGKU} are not sufficient to obtain uniqueness results for (IP). Indeed, the approach of \cite{FGKU} and the follow-up works in this direction fundamentally rely on the fact that $P_{M,g,V}$ and $\Delta_g$ commute when $V$ is identically equal to zero. This will be further elaborated on in Section~\ref{sec_literature}, along with a more thorough discussion of related inverse problems for nonlocal equations. 

To address our results on (IP), we start with a definition.
\begin{definition}
	\label{def_antipodal}
	Given any $p\in M$, we define the antipodal set of $p$, denoted by $\mathcal A_{M,g}(p)$, as the set of all points that are antipodal to $p$, that is to say,
	$$ \mathcal A_{M,g}(p) = \{q \in M\,:\, \textrm{dist}_g(p,q)= \max_{p'\in M} \textrm{dist}_{g}(p,p') \}.$$
\end{definition}
Note that since the manifold $M$ is compact, for each $p \in M$, we have $\mathcal{A}_{M,g}(p) \ne \emptyset$. 
We remark that antipodal sets have been an active area of research in differential geometry specially in the context of symmetric closed Riemannian manifolds, see e.g. \cite{CN,Ta,TaTa} as well as the survey article \cite{Chen}. Our main result regarding (IP) can now be stated as follows.
\begin{theorem}
	\label{t1}
	Let $\alpha \in (0,1)$. For $j=1,2$, let $(M_j,g_j)$ be a smooth closed and connected Riemannian manifold of dimension $n \geq 2$, and let $V_j \in C^{\infty}(M_j)$. Let $O\subset M_1\cap M_2$ be a nonempty open connected set such that $M_j\setminus\overline{O}\ne \emptyset$, $j=1,2$,   and assume that for $j=1,2$,
	\begin{itemize}
		\item [{\bf(H)}]{$(M_j\setminus O,g_j)$ is nontrapping and there exists $p\in O$ such that $\mathcal A_{M_j,g_j}(p)\subset O$.}
	\end{itemize}
	Assume also that $g_1|_{O}=g_2|_{O}$ and  $V_1|_{O}=V_2|_O=0$. Then, 
	\begin{equation}\label{thm_eq} 
		\mathcal C_{M_1,g_1,V_1}^{O} = \mathcal C_{M_2,g_2,V_2}^{O}\implies \exists\, \Phi: M_1 \stackrel{\cong}{\to} M_2 \quad \text{with  $g_1= \Phi^\star g_2$ and $V_1= V_2\circ\Phi$}.\end{equation}
	Moreover, the smooth diffeomorphism $\Phi$ is equal to identity on the set $O$.
\end{theorem}

\begin{remark}
Note that due to the obstruction to uniqueness in (IP), presented in Lemma \ref{lem_obstruction}, the implication \eqref{thm_eq} is the best possible recovery result for this problem. 
\end{remark}

An immediate consequence of the proof of Theorem \ref{t1} given in Section~\ref{sec_proofs} is the following result, which is similar to the Euclidean one in \cite{Ghosh_Lin_Xiao_2017}.
 \begin{corollary}
	\label{cor_recovery_potential}
	Let $\alpha \in (0,1)$. Let $(M,g)$ be a smooth closed and connected Riemannian manifold of dimension $n \geq 2$, and let $V_j \in C^{\infty}(M_j)$, $j=1,2$.  Let $O\subset M$ be a nonempty open set such that $M\setminus\overline{O}\ne \emptyset$.  Assume also that $V_1|_{O}=V_2|_O$. Then, 
	\[
	\mathcal C_{M,g,V_1}^{O} = \mathcal C_{M,g, V_2}^{O}\implies V_1= V_2 \text{ on } M.
		\]
\end{corollary}

Now, various corollaries of Theorem \ref{t1} can be obtained by seeking geometric conditions on the set $O$ that guarantee (H). For example, as we will show in Section~\ref{sec_proofs}, it is sufficient to impose a simplicity assumption on the submanifold $M \setminus O$, meaning that the inaccessible submanifold $(M_j \setminus O, g_j)$, $j=1,2$, is a smooth, simply connected manifold with a smooth, strictly convex boundary and no conjugate points. The simplicity assumption on $(M_j\setminus O,g_j)$, $j=1,2,$ guarantees that (H) is actually satisfied for all $p\in O$ that are sufficiently near its boundary. This leads to the following immediate corollary.

\begin{corollary}
	\label{cor1}
	Let $\alpha\in (0,1)$. For $j=1,2$, let $(M_j,g_j)$ be a smooth closed and connected Riemannian manifold of dimension $n \geq 2$, and let $V_j \in C^{\infty}(M_j)$. Let $O\subset M_1\cap M_2$ be a nonempty open connected set such that $M_j\setminus\overline{O}\ne \emptyset$, $j=1,2$. Suppose that $(M_j\setminus O,g_j)$, $j=1,2,$ is a simple Riemannian manifold, that $g_1|_{O}=g_2|_{O}$ and  $V_1|_{O}=V_2|_O=0.$ Then the implication \eqref{thm_eq} holds.
\end{corollary}
\begin{remark}
	The above mentioned corollary yields many concrete examples of the condition (H) in which the volume of the observation set $O$ could be arbitrarily small compared to the volume of $M\setminus O$. Indeed, consider any compact simple Riemannian manifold $(N,g)$ with smooth boundary $\p N$. Subsequently, consider a smooth extension of $N$ into a smooth, closed and connected manifold $M$ and also smoothly extend the metric from $N$ to $M$. Then, setting $O=M\setminus N$, condition (H) will be satisfied.
\end{remark}
\begin{remark}
	Let us also remark that in the statement of Theorem~\ref{t1} and Corollary~\ref{cor1}, the technical condition $V_1|_{O}=V_2|_{O}=0$ may be slightly weakened. Indeed, it suffices to assume that $V_1|_{O}=V_2|_{O}$ and that there exists a nonempty open set $\omega \subset O$ such that $V_1,V_2$ both vanish on $\omega$. This will become clear following the proofs.
\end{remark}

We will now present two of the main ingredients of our proof of Theorem~\ref{t1}, namely an entanglement principle and a variant of the Gel'fand inverse spectral problem. We believe that both of these results may find independent applications.

\subsection{Entanglement principle for nonlocal equations}

The (weak) unique continuation principle for linear partial differential operators can generally be stated as follows (under suitable regularity of coefficients and solutions):
\begin{itemize}
\item[{\textbf{(UCP)}}] If $Pu = 0$ on some connected open set $\Omega'$ and additionally $u$ vanishes on an open subset $\Omega \subset \Omega'$, then $u$ must also vanish on $\Omega'$.
\end{itemize}
Typical examples of operators for which this principle applies include second-order elliptic and parabolic operators with sufficiently regular coefficients, see \cite{Koch_Tataru_2001, Lerner2, Lin}.

One of the striking features of (some) nonlocal operators is the global (weak) unique continuation principles enjoyed by their solutions, which can be stated as follows (under suitable regularity of coefficients and solutions):

\begin{itemize}
	\item[{\textbf{$\text{(UCP)}^\prime$}}] If $Pu = 0$ on some open set $\Omega$ and additionally $u$ vanishes on $\Omega$, then $u$ must vanish globally.
\end{itemize}
It is clear that such unique continuation principles fail for local operators. For instance, considering $0 \ne u \in C^\infty_0(\mathbb{R}^n)$, both $u$ and $\Delta u$ vanish on a large open set while $u$ does not vanish.

On the other hand, for the fractional Laplacian $P = (-\Delta)^\alpha$ with $\alpha \in (0,1)$ on $\mathbb{R}^n$, $\textrm{(UCP)}^\prime$ is known to be true under milder assumptions on the regularity of the solution $u$; see \cite{GSU}. We refer to \cite{Riesz} for a classical result on this; see also \cite{FF14, Ruland_2015, Yu, Ghosh_Uhlmann_2021}. The $\textrm{(UCP)}^\prime$ was further extended in \cite{Ghosh_Lin_Xiao_2017} to the case of the fractional Laplace--Beltrami operator $P = (-\Delta_g)^\alpha$, $\alpha \in (0,1)$, on $\mathbb{R}^n$ with a smooth Riemannian metric. The $\textrm{(UCP)}^\prime$ for the higher-order fractional Laplacian $(- \Delta)^\alpha$ when $\alpha \in (-n/2, \infty) \setminus \mathbb{Z}$ on $\mathbb{R}^n$ also holds; see \cite{CMR_2022}. We mention that the proofs of these $\textrm{(UCP)}^\prime$ results rely on the Caffarelli--Silvestre extension \cite{CS} and Carleman estimates from \cite{Ruland_2015}.

Let us also emphasize that the literature on inverse problems for nonlocal equations dealing with the unique recovery of lower-order coefficients have so far heavily relied on $\textrm{(UCP)}^\prime$. This is essentially due to a density argument that follows from $\textrm{(UCP)}^\prime$, as introduced first in the work of \cite{GSU}, where the authors uniquely recover an unknown compactly supported potential $V$ from the exterior data of solutions to $(- \Delta)^\alpha u + V u = 0$ on a domain $\Omega \subset \mathbb{R}^n$. It is the same principle that allowed the authors in \cite{Ghosh_Lin_Xiao_2017} to recover the zeroth-order coefficient $V$ from the Cauchy data set of solutions to $(- \Delta_g)^\alpha u + V u = 0$ over some domain $\Omega \subset \mathbb{R}^n$. Interestingly, the results in \cite{Fei21, FGKU, CO, Ruland_2023}, which deal with the recovery of leading-order coefficients in nonlocal equations of the form $(- \Delta_g)^\alpha u = 0$ on a compact Riemannian manifold $M$, are based on different techniques and do not use $(\textrm{UCP})^\prime$.

In this paper, as part of our solution to (IP) on closed manifolds, we discover yet another striking feature of nonlocal equations that can be viewed as an even stronger version of $\textrm{(UCP)}^\prime$. The principle essentially states that if different fractional powers of the Laplace--Beltrami operator acting on several distinct functions, which vanish on some open set $O$, are known to be linearly dependent on $O$, then all the functions must be zero. Precisely, we prove the following theorem in Section~\ref{sec_ent}. Here, the fractional powers of the Laplace--Beltrami operator are defined using the functional calculus for self-adjoint operators.

\begin{theorem}
	\label{thm_ent}
Let $(M,g)$ be a smooth closed and connected Riemannian manifold of dimension $n \geq 2$, and let $O \subset M$ be a nonempty open set. Let $N \in \mathbb{N}$, and let $\{v_j\}_{j=1}^N \subset C^\infty(M)$ satisfy
\begin{equation}\label{v_12_O} 
v_1|_O = \ldots = v_N|_O = 0 \quad \text{and} \quad \sum_{j=1}^N b_j((-\Delta_g)^{\alpha_j} v_j)|_O = 0,
\end{equation}
for some $b_1, \ldots, b_N\in \C\setminus\{0\}$, and some $\{\alpha_j\}_{j=1}^N \subset (0, \infty) \setminus \mathbb{N}$ that additionally satisfy
\begin{equation}\label{diff_int} 
\alpha_j - \alpha_k \notin \mathbb{Z} \quad \text{for all } j,k = 1, \ldots, N \text{ with } j \neq k.
\end{equation}
Then $v_j = 0$ on $M$ for all $j = 1, \ldots, N$.
\end{theorem} 

We remark that the case $N=1$ is equivalent to establishing $\textrm{(UCP)}^\prime$ on a closed manifold, but the cases $N \geq 2$ can be seen as a stronger version of $\textrm{(UCP)}^\prime$ that simultaneously involves several functions. We call this the entanglement principle for nonlocal operators. 

\begin{remark}
In the case where $N \ge 2$, our proof of Theorem \ref{thm_ent} relies crucially on the compactness of the manifold $M$, which implies that the heat semigroup $e^{t\Delta_g}$, $t> 0$, when acting on the orthogonal complement of the one-dimensional subspace of $L^2(M)$ spanned by the constant function $1$, exhibits exponential decay for large times; see the bound \eqref{heat_kernel_bound_2} below. 
\end{remark}

\begin{remark}
Let us also emphasize that the condition \eqref{diff_int} in the theorem is sharp. Indeed, consider the case $N=2$ with $\alpha_1 = \alpha \in (0,\infty) \setminus \mathbb{N}$ and $\alpha_2 = k + \alpha$ for some $k \in \mathbb{N}$. Then, given any nonzero function $v \in C^\infty(M)$ satisfying $v|_O = 0$, it is trivial to see that for the nonzero functions $v_1 = (-\Delta_g)^k v$ and $v_2 = v$, there holds
\[
(-\Delta_g)^{\alpha_1} v_1 - (-\Delta_g)^{\alpha_2} v_2 = 0 \quad \text{on } M.
\]
Such counterexamples can be generalized to various values of $N \geq 2$.
\end{remark}

Our proof of Theorem~\ref{thm_ent} will essentially be based on two key ingredients. One ingredient is the representation of fractional operators in terms of heat semigroups, together with precise bounds on the heat semigroup and unique continuation principles for heat equations. Another key ingredient is related to interpolation properties for holomorphic functions subject to certain growth rates at infinity. It would be interesting to see whether an alternative proof of Theorem~\ref{thm_ent} can be obtained via the Caffarelli--Silvestre extension procedure.

\subsection{Variant of Gel'fand inverse spectral problem} 

Inverse spectral problems have a long and distinguished history, dating back to the classical works \cite{Borg_1946, Levinson_1949, Gelfand_Levitan_1951}, see also \cite{Gelfand}. For Schr\"odinger operators on bounded domains in Euclidean space, the inverse spectral problem of recovering a potential from the knowledge of the Dirichlet eigenvalues and the boundary traces of the normal derivatives of the normalized Dirichlet eigenfunctions was solved in \cite{Nachman_Sylveser_Uhlmann_88} and \cite{Nov_1988}.

A more general version of Gel'fand's inverse spectral problem is formulated on a Riemannian manifold, with or without boundary. This problem concerns the determination not only of coefficients of partial differential operators but also of the Riemannian manifold itself from some knowledge of its {\em normalized} spectral data.  For example, as studied in \cite{BKL, KrKaLa, HLOS_2018}, in the context of a closed manifold, the spectral data can be formulated as eigenvalues and restrictions to the set $O$ of an $L^2(M)$-orthonormal basis consisting of eigenfunctions of $-\Delta_g$ on $M$. We also refer to \cite{AKKLT, HLOS_2018, HLYZ_2020, Kurylev_Lassas_Yamaguchi, FILLN} for further works in the context of inverse interior spectral problems for wave and heat equations.

Such inverse spectral problems play a major role in the study of inverse problems for evolution equations with the so-called active measurements, where one is allowed to perform various experiments and observe the data. For example, in geophysical exploration, an important problem is to uniquely determine the internal structure of the Earth by sending artificial packets of seismic or electromagnetic waves of a fixed frequency into the Earth and subsequently measuring the reflected waves that arrive back on its surface. This is an inverse problem where the governing equation is the wave equation. Similar inverse problems are studied for heat equations, Schr\"odinger equations, and other evolutionary PDEs. In fact, all these problems have been shown to be equivalent \cite{KKLM}.

Gel'fand's inverse spectral problem for the acoustic operator $-c^2(x)\Delta$ was first solved in the pioneering work of Belishev \cite{Bel1} on domains in $\mathbb{R}^n$, $n \geq 2$, and subsequently by Belishev and Kurylev for the Laplace--Beltrami operator on general Riemannian manifolds \cite{BK92}.  The solution relies on the construction of a suitable control theory from the boundary for solutions to the wave equation that is commonly known as the Boundary Control (BC) method. The BC method also fundamentally relies on a sharp unique continuation theorem for wave equations that was developed by Tataru \cite{Tataru} (see also the important precursor by Robbiano \cite{R1}, and related later results by Robbiano and Zuily \cite{RZ} and Tataru \cite{Tataru2}). We refer the reader to \cite{Bel2, KKL} for an exposition of the BC method.

The inverse problem (IP) studied in this manuscript can be reduced, without imposing any geometric assumption, to a variant of Gel'fand's inverse spectral problem on $(M,g)$. Specifically, as we will show in Proposition~\ref{prop1}, the knowledge of $\mathcal{C}_{M,g,V}^O$ uniquely determines the spectral data consisting of eigenvalues and restrictions to the set $O$ of an $L^2(M)$ Schauder basis consisting of eigenfunctions of $-\Delta_g$ on $M$. The only difference here is that the eigenfunctions are not necessarily orthonormal. This poses a significant change in the problem as the solution of Gel'fand's inverse spectral problem fundamentally relies on its immediate reduction to inverse problems for wave equations. When the eigenfunctions are not normalized, such a connection is not known. In this paper, we provide a resolution of the problem under the geometric assumption (H). Specifically, we prove the following theorem in Section~\ref{sec_Gelfand}.

\begin{theorem}
	\label{t2}
	For $j=1,2$, let $(M_j,g_j)$ be a smooth closed and connected Riemannian manifold of dimension $n \geq 2$. Let $O \subset M_1 \cap M_2$ be a nonempty open set such that $M_j\setminus\overline{O}\ne \emptyset$, and assume that (H) is satisfied for $j=1,2$. Assume also that $g_1|_{O} = g_2|_{O}$. For $j=1,2$, suppose that there exists an $L^2(M_j)$ Schauder basis consisting of eigenfunctions $\{\psi_k^{(j)}\}_{k=0}^{\infty} \subset C^{\infty}(M_j)$ for $-\Delta_{g_j}$ on $(M_j,g_j)$ corresponding to (not necessarily distinct) eigenvalues 
\[0 = \mu_0^{(j)} < \mu_1^{(j)} \leq \mu_2^{(j)} \leq \mu_3^{(j)} \leq \ldots\]
such that, given any $k=0,1,2,\ldots$, there holds
\begin{equation}\label{psi_eq} \mu_k^{(1)} = \mu_k^{(2)} \quad \text{and} \quad \psi_k^{(1)}(x) = \psi_k^{(2)}(x) \quad \forall \, x \in O. \end{equation} 
Then, there exists a smooth diffeomorphism $\Phi: M_1 \to M_2$ that is identity on $O$ such that $g_1 = \Phi^\star g_2$ on $M_1$.
\end{theorem}

The crux of the proof of Theorem~\ref{t2} is showing that under the assumption (H), the non-normalized spectral data determines the normalized spectral data. This is accomplished via establishing a hidden connection between non-normalized spectral data and wave equations. To accomplish this, we will fundamentally rely on the sharp unique continuation result of Tataru applied to antipodal sets, as well as requiring global control of energy of eigenfunctions based on their local energy in $O$. The latter step is in particular responsible for us imposing the nontrapping assumption on $M\setminus O$ that is related to the controllability results of Bardos, Lebeau, and Rauch \cite{BLR0,BLR}. We mention that this nontrapping condition in (H) can be relaxed by assuming that the following more abstract condition is satisfied,
\begin{itemize}
	\item [{\bf (C)}]{There exists a constant $C>0$ depending on $(M,g)$ and $O$ such that given any eigenfunction $\phi$ of $-\Delta_g$ on $M$, there holds:
	$$ \|\phi\|_{L^2(M)} \leq C \|\phi\|_{L^2(O)}.$$}
\end{itemize}
Such local energy controls for eigenfunctions have been an active topic of research for many years due to their applications in control theory problems for evolutionary equations, as well as connections to quantum chaos, see e.g. \cite{DJN}. In particular, let us mention that the assumption (C) above is satisfied in dimension two on any closed Riemannian surface whose geodesic flow satisfies the Anosov property, see \cite[Theorem 1]{DJN}. For examples of higher dimensional closed Riemannian manifolds that admit some trapping of geodesics on the set $M\setminus O$ and still satisfy the condition (C), we refer the reader to the work of Anantharaman--Rivi\`ere \cite{AR}.

We close this section by mentioning that the idea of requiring control on the local energy of eigenfunctions has appeared before in the literature of inverse problems (albeit for different reasons), see e.g. \cite{LNOY,LO1,LO2} for the study of inverse problems for waves subject to disjoint data measurements. The paper \cite{LO1} uses controls for eigenfunctions on Riemannian manifolds with boundary, stemming from the work of Hassell and Tao in \cite{HT}.

\subsection{Previous literature}\label{sec_literature}
The study of the fractional Calder\'on problem was initiated in \cite{GSU} where the authors considered the inverse problem of recovering an unknown potential in the fractional Schr\"odinger equation $(-\Delta)^\alpha u + Vu=0$ on a bounded domain in the Euclidean space given the knowledge of exterior Cauchy data measurements analogous to the data \eqref{DN_map}. Following this work, there have been an extensive literature dealing with recovery of lower order coefficients as well as isotropic leading order coefficients for fractional elliptic equations, see for example \cite{Ghosh_Ruland_Salo_Uhlmann_2020}, \cite{Ghosh_Lin_Xiao_2017}, \cite{Ruland_Salo_2020}, \cite{Ruland_Salo_2018}, \cite{BGU}, \cite{Ch}, \cite{CLR}, \cite{Co1}, \cite{CGR}, \cite{CMRU}, \cite{Li_Li_2020}, \cite{Li_Li_2021},  \cite{Ruland_2021}, and \cite{BCR} for some of the important contributions. 

Recently, there have been uniqueness results for recovery of anisotropic leading order coefficients appearing in nonlocal equations that are of the specific form $L^\alpha u=0$ with $L$ a nonnegative self-adjoint elliptic differential operator, based on two fundamentally different approaches yielding different types of results. The works \cite{Ghosh_Uhlmann_2021,CGRU} introduced an approach for the study of inverse problems for fractional powers of self-adjoint elliptic operators, $L^\alpha$, that is based on reducing the inverse problem to an inverse problem for the local equation $Lu=0$. This has allowed the authors to obtain certain uniqueness results for the nonlocal problem by resorting to known uniqueness results for inverse problems associated to the local equation $Lu=0$ in the literature. We also mention the recent works \cite{Lin23, LLU} that applied an analogous idea in the study of fractional parabolic operators.

An altogether different approach was discovered in the recent works \cite{Fei21,FGKU} for solving (IP) under the assumption that $V$ is identically zero. The work \cite{Fei21} solves the inverse problem under a certain Gevrey analyticity assumption on the observation set $O$, and is based on a reduction from (IP) (with $V=0$) to the Gel'fand inverse spectral problem. The work \cite{FGKU} gives a complete resolution to this inverse problem (again assuming that $V$ is identically zero) without imposing any geometric assumptions on the manifold. The idea in \cite{FGKU} is to use the fact that the local operator $-\Delta_g$ commutes with $(-\Delta_g)^\alpha$ to obtain the heat semigroup $e^{t\Delta_g}(x,y)$, with $t> 0$ and $x,y \in O$ and subsequently uses Kannai's trasmutation formula \cite{Ka} to reduce the inverse problem to a well understood inverse problem for waves \cite{BK92,HLOS_2018}. The subsequent work \cite{CO} obtains a similar uniqueness result for noncompact complete Riemannian manifolds and the work \cite{Ruland_2023} provides a different perspective to the approach of \cite{Fei21,FGKU} via the Caffarelli--Silvestre extension for nonlocal equations. We mention also the works \cite{Chien,QU} for similar results obtained in the case of connection Laplacian and the Dirac operator on closed manifolds. 

In all the previous works on recovery of leading order coefficients, the fact that the operator is of the form $L^\alpha$ plays a crucial role, owing to the fact that $L^\alpha$ and $L$ commute, thanks to the calculus of powers of self-adjoint operators. Extension of such methods to operators of the form $L^\alpha +V$ does not seem to work and new ideas are needed. To the best of our knowledge, the problem of recovering both the manifold $(M,g)$ and lower order coefficients has remained open until now, with Theorem~\ref{t1} providing a first uniqueness result in this direction, under the assumption (H). We close this section by emphasizing that any future improvements to Theorem~\ref{t2} (in the sense of weakening the assumption (H)) will yield immediate improvements to Theorem~\ref{t1}. This is due to the fact that the condition (H) is only needed in our proof of Theorem~\ref{t2}, namely the variant of Gel'fand's inverse spectral problem. 

\subsection{Organization of the paper}
We begin the paper by fixing some notations in Section~\ref{sec_pre} and discussing the forward problem for the nonlocal Schr\"odinger equation. In Section~\ref{sec_ent}, we present the proof of the entanglement principle using ideas from complex analysis. We will also state and prove Lemma~\ref{density_lemma}, which will be needed in the subsequent section. Section~\ref{sec_reduction_to_spectral_data} focuses on showing that the Cauchy data set $\mathcal{C}_{M,g,V}^O$ uniquely determines the restrictions on the set $O$ of a non-normalized Schauder basis consisting of eigenfunctions of $(M,g)$ associated with $-\Delta_g$. This reduces the question of recovering the geometry $(M,g)$ to a variant of the Gel'fand inverse spectral problem. We present the proof of this variant in Section~\ref{sec_Gelfand}, emphasizing that this is the only step of the proof that imposes condition (H). We address the recovery of the lower-order coefficient $V$ in Section~\ref{sec_proofs} and present the proof of Corollary~\ref{cor1}. Finally, in Appendix~\ref{app_obstruction}, we discuss an obstruction to uniqueness for the inverse problem (IP).

\section{Preliminaries}\label{sec_pre}

Let $(M, g)$ be a smooth closed and connected Riemannian manifold of dimension $n \ge 2$. Let $-\Delta_g$ be the positive Laplace--Beltrami operator on $M$. This operator is self-adjoint on $L^2(M)$, with the domain $\mathcal{D}(-\Delta_g) = H^2(M)$, the standard Sobolev space on $M$. We denote by
$$0 = \lambda_0 < \lambda_1 < \lambda_2 < \ldots$$ 
the distinct eigenvalues of the Laplace--Beltrami operator $-\Delta_g$, and let $d_k$ be the multiplicity of the eigenvalue $\lambda_k$ for $k = 0, 1, 2, \ldots$. Let $\phi_{k,1}, \ldots, \phi_{k,d_k}$ be an $L^2(M)$-orthonormal basis for the eigenspace $\textrm{Ker}(-\Delta_g - \lambda_k)$ corresponding to $\lambda_k$. Finally, given any $k = 0, 1, \ldots$, we define $\pi_k : L^2(M) \to \textrm{Ker}(-\Delta_g - \lambda_k)$ to be the orthogonal projection operator onto the eigenspace of $\lambda_k$, defined via
\begin{equation}
	\label{proj_op}
	\pi_k f = \sum_{\ell=1}^{d_k} (f, \phi_{k,\ell})_{L^2(M)} \, \phi_{k,\ell},\quad f\in L^2(M),
\end{equation}
where $(\cdot, \cdot)_{L^2(M)}$ is the $L^2$ inner product on $M$.

Let $\alpha > 0$. By the spectral theorem, we define the fractional Laplacian $(-\Delta_g)^\alpha$ of order $\alpha$ as an unbounded self-adjoint operator on $L^2(M)$ given by
\begin{equation}
	\label{alpha_laplace}
	(-\Delta_g)^{\alpha}u = \sum_{k=0}^{\infty} \lambda_k^{\alpha} \, \pi_k u, 
\end{equation}
equipped with the domain $\mathcal{D}((-\Delta_g)^\alpha) := \left\{ u \in L^2(M) : \sum_{k=0}^\infty \lambda_k^{2\alpha} \|\pi_k u\|_{L^2(M)}^2 < \infty \right\} = H^{2\alpha}(M)$; see \cite[page 25]{Taylor_book_1981}.

From now on, let us assume that $\alpha \in (0,1)$. Let $f \in C^\infty(M)$ be such that 
$(f,1)_{L^2(M)} = 0$.
Then the equation 
\begin{equation}
\label{eq_int_2}
(-\Delta_g)^\alpha u = f \quad \text{in} \quad M
\end{equation}
has a unique solution $u = u^f \in C^\infty(M)$ with the property that $(u^f, 1)_{L^2(M)} = 0$, given by 
\[
u^f = (-\Delta_g)^{-\alpha} f = \sum_{k=1}^\infty \lambda_k^{-\alpha} \pi_k f,
\]
see \cite[Corollary 2.39]{McLean_book}.

It is useful to express the self-adjoint operator $(-\Delta_g)^{-\alpha}$ in terms of the heat semigroup $e^{t\Delta_g}$, $t \ge 0$, when acting on the orthogonal complement of the one-dimensional subspace of $L^2(M)$ spanned by the constant function $1$. To achieve this, we use the definition of the Gamma function:
\[
\Gamma(\alpha) = \int_0^\infty e^{-t} t^{\alpha-1} \, dt,
\]
which gives us
\begin{equation}
	\label{eq_2_1}
	a^{-\alpha} = \frac{1}{\Gamma(\alpha)} \int_0^\infty e^{-at} \frac{1}{t^{1-\alpha}} \, dt, \quad a > 0.
\end{equation}
From \eqref{eq_2_1}, it follows that 
\begin{equation}
	\label{frac_laplace_heat_inverse}
	(-\Delta_g)^{-\alpha} v = \frac{1}{\Gamma(\alpha)} \int_0^\infty e^{t\Delta_g} v \frac{1}{t^{1-\alpha}} \, dt,
\end{equation}
where $v \in L^2(M)$ satisfies $(v, 1)_{L^2(M)} = 0$. We note that the integral in \eqref{frac_laplace_heat_inverse} converges in $L^2(M)$ due to the estimate
\begin{equation}
\label{heat_kernel_bound_2}
\|e^{t\Delta_g} v\|_{L^2(M)} \le e^{-\beta t} \|v\|_{L^2(M)}, \quad t \ge 0,
\end{equation}
for some $\beta > 0$. 
The estimate \eqref{heat_kernel_bound_2} follows from the expansion $e^{t\Delta_g} v = \sum_{k=1}^\infty e^{-t\lambda_k} \pi_k v$ and the fact that $0 < \lambda_1 < \lambda_2 < \cdots$.

Recalling that $\alpha \in (0,1)$, one derives from \eqref{eq_2_1} the following identity:
\[
a^{\alpha} = \frac{1}{\Gamma(-\alpha)} \int_0^\infty (e^{-at} - 1) \frac{dt}{t^{1+\alpha}},
\]
valid for all $a \ge 0$. In view of this identity, and by the functional calculus, we get 
\begin{equation}
\label{frac_laplace_heat}
(-\Delta_g)^\alpha u = \frac{1}{\Gamma(-\alpha)} \int_0^\infty (e^{t\Delta_g} - 1)u \frac{dt}{t^{1+\alpha}},
\end{equation}
for $u \in \mathcal{D}(-\Delta_g) = H^2(M)$; see \cite{Caffarelli_Stinga_2016}. The integral on the right-hand side of \eqref{frac_laplace_heat} converges in $L^2(M)$.

For future reference, we state the following pointwise upper Gaussian estimate on the heat kernel $e^{t\Delta_g}(x,y) \in C^\infty((0,\infty) \times M \times M)$; see \cite[Theorem 4.7, page 171]{SY} and also \cite{Gr, Va}:
\begin{equation}
	\label{Gauss_heat}
	|e^{t\Delta_{g}}(x,y)| \leq C\, t^{-\frac{n}{2}}\, e^{-c \frac{(\mathrm{dist}_{g}(x,y))^2}{t}}, \quad t \in (0,1), \quad x,y \in M,
\end{equation}
where $c>0$ and $C>0$, and $\mathrm{dist}_{g}(\cdot, \cdot)$ denotes the Riemannian distance on $(M,g)$.

We shall also need the following equivalence of norms in the Sobolev space $H^s(M)$, $s \in \mathbb{R}$; see \cite[Section 4.2.3, page 103]{Sogge_book_eigenfunctions}:
\begin{equation}
\label{eq_prelim_Sobolev_norms}
c_s \|(I-\Delta_g)^{s/2} u\|_{L^2(M)} \le \|u\|_{H^s(M)} \le C_s \|(I-\Delta_g)^{s/2} u\|_{L^2(M)}, \quad C_s, c_s > 0.
\end{equation}
Here, the Bessel potential $(I-\Delta_g)^{s/2}$ is defined by the self-adjoint functional calculus. 

For future reference, we record the following fact. When $u \in C^\infty(M)$, we have $e^{t\Delta_g} u \in C^\infty([0,\infty); H^s(M))$ for any $s \in \mathbb{R}$, implying $e^{t\Delta_g}  u \in C^\infty([0,\infty);C^\infty(M))$. Fixing $s > n/2$ such that $s/2 \in \mathbb{N}$, by Sobolev's embedding theorem and \eqref{heat_kernel_bound_2}, we get for $t > 0$,
\begin{equation}
\label{eq_100_-1}
\begin{aligned}
\|e^{t\Delta_g} \Delta_g u \|_{L^\infty(M)} \leq C \|e^{t\Delta_g} \Delta_g u \|_{H^s(M)} 
&\leq C \|e^{t\Delta_g} (1 - \Delta_g)^{s/2} \Delta_g u \|_{L^2(M)} \\
&\leq C e^{-\beta t} \|(1 - \Delta_g)^{s/2} \Delta_g u \|_{L^2(M)}.
\end{aligned}
\end{equation}

For further reference, we note that the following bound holds for $t > 0$,
\begin{equation}
\label{eq_800_1}
\|e^{t\Delta_g}\|_{L^\infty(M) \to L^\infty(M)} \le 1,
\end{equation}
see \cite[Theorem 3.5]{St}.

We shall denote by $\Psi^\mu_{\text{cl}}(M)$ the space of classical pseudodifferential operators of order $\mu$ on $M$. According to Seeley's theorem, we have $(-\Delta_g)^\alpha \in \Psi^{2\alpha}_{\text{cl}}(M)$; see \cite{Seeley_1966}. Consequently, $(-\Delta_g)^\alpha: C^\infty(M) \to C^\infty(M)$, and $(-\Delta_g)^\alpha: H^s(M) \to H^{s-2\alpha}(M)$ is bounded for all $s \in \mathbb{R}$; see \cite[Chapter 7, Section 10]{Taylor_book_vol_II} or \cite[Theorem 8.5, page 206]{Grubb_book}.

We shall also need the following result, which states that when $u \in C^\infty(M)$, the equality \eqref{frac_laplace_heat} holds pointwise.
\begin{proposition}
Let $\alpha \in (0,1)$ and let $u \in C^\infty(M)$. Then we have the following pointwise formula:
\begin{equation}
\label{frac_laplace_heat_pointwise}
((-\Delta_g)^\alpha u)(x) = \frac{1}{\Gamma(-\alpha)} \int_0^\infty ((e^{t\Delta_g}u)(x) - u(x)) \frac{dt}{t^{1+\alpha}}, \quad \text{for all } x \in M.
\end{equation}
\end{proposition}
\begin{proof}
First, we claim that for all $x \in M$, the function
\begin{equation}
\label{eq_800_2}
t \mapsto \frac{(e^{t\Delta_g}u)(x) - u(x)}{t^{1+\alpha}} \in L^1((0, \infty)).
\end{equation}
Indeed, for $t > 1$ and $x \in M$, in view of \eqref{eq_800_1}, we get
\begin{equation}
\label{eq_800_3}
|(e^{t\Delta_g}u)(x) - u(x)| \le 2\|u\|_{L^\infty(M)}.
\end{equation}
When $0 < t < 1$ and $x \in M$, recalling that $e^{t\Delta_g} u \in C^\infty([0, \infty); C^\infty(M))$ and using the fundamental theorem of calculus, we get
\begin{equation}
\label{eq_800_4}
(e^{t\Delta_g}u)(x) - u(x) = \int_0^t (e^{s\Delta_g}\Delta_g u)(x) \, ds.
\end{equation}
Using \eqref{eq_800_1}, we obtain from \eqref{eq_800_4} that
\begin{equation}
\label{eq_800_5}
|(e^{t\Delta_g}u)(x) - u(x)| \le t \|\Delta_g u\|_{L^\infty(M)}.
\end{equation}
The claim \eqref{eq_800_2} follows from \eqref{eq_800_3} and \eqref{eq_800_5}.

Now, since $u \in \mathcal{D}(-\Delta_g)$, by the functional calculus, \eqref{frac_laplace_heat_pointwise} holds in $L^2(M)$, see \eqref{frac_laplace_heat}. The function on the left-hand side of \eqref{frac_laplace_heat_pointwise} is in $C^\infty(M)$, and the function on the right-hand side of \eqref{frac_laplace_heat_pointwise} is in $C(M)$. The latter can be seen by the dominated convergence theorem relying on  \eqref{eq_800_2},  \eqref{eq_800_3}, and \eqref{eq_800_5}. Hence, \eqref{frac_laplace_heat_pointwise} holds pointwise for all $x \in M$.
\end{proof}

We will next state a classical result for the solvability of the inhomogeneous fractional Schrödinger equation \eqref{pf} below, which demonstrates that the Cauchy data set $\mathcal C^O_{M,g,V}$, defined in \eqref{DN_map}, possesses a rich structure. The proof of this result follows, for example, by applying \cite[Theorem 2.34]{McLean_book}. For the sake of completeness and the reader's convenience, we present a short proof based on the Lax–Milgram lemma.
\begin{proposition}
	\label{prop_direct}
	Let $V \in C^{\infty}(M)$. Then the operator $(-\Delta_g)^\alpha + V: H^\alpha(M) \to H^{-\alpha}(M)$ is Fredholm of index zero, and hence, the Fredholm alternative holds for the equation
\begin{equation}
\label{pf}
P_{M,g,V}\, u := (-\Delta_g)^{\alpha} u + V u = f \qquad \text{on} \quad M.
\end{equation}
Specifically, letting
\[
\mathcal{K}_{M,g,V} := \{ u \in H^{\alpha}(M) : P_{M,g,V}\, u = 0 \quad \text{on} \quad M \},
\]
there are two mutually exclusive possibilities:
\begin{itemize}
\item[(i)] $\mathcal{K}_{M,g,V} = \{0\}$. In this case, for each $f \in H^{-\alpha}(M)$, the inhomogeneous equation \eqref{pf} has a unique solution $u \in H^\alpha(M)$.
\item[(ii)] $\dim(\mathcal{K}_{M,g,V}) = N$, $1 \le N < \infty$. In this case, given $f \in H^{-\alpha}(M)$, the inhomogeneous equation \eqref{pf} is solvable if and only if $(f, v)_{H^{-\alpha}(M), H^\alpha(M)} = 0$ for all $v \in \mathcal{K}_{M,g,\overline{V}}$.
\end{itemize}
\end{proposition} 

\begin{proof}
First, we will demonstrate that the operator $(-\Delta_g)^{\alpha} + 1 : H^{\alpha}(M) \to H^{-\alpha}(M)$ is invertible. To this end, let $u, v \in C^\infty(M)$ and consider the sesquilinear form associated with this operator:
\[
a(u,v) = ((-\Delta_g)^\alpha u, v)_{L^2(M)} + (u, v)_{L^2(M)} = ((-\Delta_g)^{\alpha/2} u, (-\Delta_g)^{\alpha/2} v)_{L^2(M)} + (u, v)_{L^2(M)}.
\]
The second equality follows from functional calculus. Using the Cauchy–Schwarz inequality and \cite[Theorem 4.4]{St} regarding the equivalence of Sobolev norms, we obtain
\begin{align*}
|a(u,v)| &\leq \|(-\Delta_g)^{\alpha/2} u\|_{L^2(M)} \|(-\Delta_g)^{\alpha/2} v\|_{L^2(M)} + \|u\|_{L^2(M)} \|v\|_{L^2(M)} \\
&\leq C \|u\|_{H^\alpha(M)} \|v\|_{H^\alpha(M)},
\end{align*}
for all $u, v \in C^\infty(M)$. Thus, the form $(u, v) \mapsto a(u, v)$ extends uniquely to a continuous sesquilinear form on $H^\alpha(M) \times H^\alpha(M)$. The form $a$ is coercive on $H^\alpha(M)$, since
\[
a(u, u) = \|(-\Delta_g)^{\alpha/2} u\|^2_{L^2(M)} + \|u\|^2_{L^2(M)} \geq c \| u\|^2_{H^\alpha(M)},
\]
for all $u \in H^\alpha(M)$ thanks to the equivalence of the Sobolev norms, see \cite[Theorem 4.4]{St}. Here, $c > 0$. By the Lax–Milgram lemma, the operator $(-\Delta_g)^{\alpha} + 1 : H^{\alpha}(M) \to H^{-\alpha}(M)$ is invertible.

The operator of multiplication by $(V - 1) \in C^\infty(M)$ is compact from $H^{\alpha}(M)$ to $H^{-\alpha}(M)$, see \cite[Theorems 2.3.6 and 2.3.1]{Agranovich_book}. Hence, the operator $(-\Delta_g)^{\alpha} + V : H^\alpha(M) \to H^{-\alpha}(M)$ is Fredholm of index zero, and the result follows.
\end{proof}

\begin{remark}
\label{rem_direct}
As $(-\Delta_g)^\alpha \in \Psi^{2\alpha}_{\text{cl}}(M)$ and $V \in C^\infty(M)$, by elliptic regularity, we see that $\mathcal{K}_{M,g,V} \subset C^\infty(M)$. Furthermore, if $f\in C^\infty(M)$, then any solution $u$ to \eqref{pf} is in $C^\infty(M)$; see  \cite[Chapter 7, Section 10]{Taylor_book_vol_II}.

\end{remark}

For future reference, we shall also record the following standard fact, which is a consequence of Proposition \ref{prop_direct} and Remark \ref{rem_direct}, together with the orthogonal decomposition $H^\alpha(M) = \mathcal{K}_{M,g,V} \oplus (\mathcal{K}_{M,g,V})^\perp$.
\begin{proposition}
		\label{prop_direct_2}
Assume that $\dim(\mathcal{K}_{M,g,V}) = N$, with $1 \le N < \infty$. Then, given any
\[
    f \in \mathcal{H}_{M,g,V} := \{ f \in C^{\infty}(M) : (f, v)_{L^2(M)} = 0 \quad \forall\, v \in \mathcal{K}_{M,g,\overline{V}} \},
\]
the equation \eqref{pf} admits a unique solution $u \in C^{\infty}(M)$ subject to the additional orthogonality condition
\begin{equation}\label{ortho_cond_S}
    (u, v)_{L^2(M)} = 0 \quad \forall\, v \in \mathcal{K}_{M,g,V}.
\end{equation}
\end{proposition}

In view of the above proposition, let us state the following definition.
\begin{definition}\label{def_S} 
Let $O \subset M$ be an open, nonempty set such that $M\setminus \overline O\neq \emptyset$. Assume first that $\dim(\mathcal{K}_{M,g,V}) = N$, with $1 \leq N < \infty$. Define
\begin{equation}
    \label{def_H}
    \mathcal{H}^O_{M,g,V} := \{ f \in C^{\infty}_0(O) : (f, v)_{L^2(M)} = 0 \quad \forall v \in \mathcal{K}_{M,g,\overline{V}} \},
\end{equation}
and for each $f \in \mathcal{H}^O_{M,g,V}$, let $u = S_{M,g,V}(f) \in C^{\infty}(M)$ be the unique solution to \eqref{pf} subject to the orthogonality condition \eqref{ortho_cond_S}.

In the case when $\mathcal{K}_{M,g,V} = \{0\}$, we define $\mathcal{H}^O_{M,g,V} = C^{\infty}_0(O)$, and also $u = S_{M,g,V}(f) \in C^{\infty}(M)$ to be the unique solution to \eqref{pf}.
\end{definition}

\begin{remark}
\label{rem_structure_H_set}
Assume that $\dim(\mathcal{K}_{M,g,V}) = N$, where $1 \le N < \infty$. Let $\zeta_1, \dots, \zeta_N \in \mathcal{K}_{M,g,\overline{V}}$ be a collection of linearly independent functions.  Then by Lemma \ref{lem_analog_5_2} below, there are functions $\{\theta_k\}_{k=1}^N \subset C^\infty_0(O)$ such that $(\theta_k, \zeta_l)_{L^2(O)} = \delta_{lk}$ for all $k, l = 1, \dots, N$. Here $\delta_{lk}$ is the Kronecker delta function. Thus, letting $f \in C_0^\infty(O)$ be arbitrary and setting
$$
\tilde{f} = f - \sum_{j=1}^N (f, \zeta_j)_{L^2(O)} \theta_j,
$$
we observe that $\tilde{f} \in \mathcal H^O_{M,g,V}$.
\end{remark}

\section{Entanglement principle for nonlocal equations}
\label{sec_ent}

The main goal of this section is to prove the entanglement principle discussed in Theorem~\ref{thm_ent}. We recall that the main content of the theorem lies in the case $N\geq 2$, as the case $N=1$ is equivalent to the unique continuation principle for the fractional Laplacian. Since we have not been able to locate a proof of the unique continuation principle for the fractional Laplacian on closed manifolds, we mention that the simplified version of our proof provides this,  see Remark~\ref{rem_UCP_N_1} below.

The proof of Theorem~\ref{thm_ent} will rely on some sharp interpolation properties for holomorphic functions in the complex plane that is in the same spirit as the theorem of Carlson in complex analysis \cite{Carlson}. The version that we need here is due to Pila \cite{Pila} based on the asymptotic behaviour of the Gamma function. Theorem~\ref{thm_ent} follows from a combination of such interpolation properties for holomorphic functions combined with unique continuation principles for heat equations. We will first start with the following proposition. 

\begin{proposition}
	\label{prop_interpolate}
		Let $N\in \N$ and let $\{\alpha_j\}_{j=1}^N \subset (0,\infty)$ satisfy \eqref{diff_int}. Assume that  $\{f_j\}_{j=1}^N \subset C^{\infty}((0,\infty))$ and that there exists  constants $c>0$ and $\delta>0$ such that for each $j=1,\ldots,N$, the function $f=f_j$ satisfies the bounds
	\begin{equation}
		\label{f_bounds}
		|f (t)| \leq c\, e^{-\delta t}, \quad  t\in [1,\infty), \quad \text{and} \quad |f(t)|\leq c\,e^{-\,\frac{\delta}{t}}, \quad t \in (0,1].
		\end{equation}
If furthermore there exists $l\in \N$ such that 
	\begin{equation}
		\label{f_density}
		\sum_{j=1}^N \Gamma(m+1+\alpha_j)\, \int_0^{\infty} f_j(t)\, t^{-m}\,dt=0 \quad \text{for all $m=l,l+1,l+2,\dots$},
	\end{equation}
then $f_j(t)=0$ for all $t\in (0,\infty)$ and all $j=1,\ldots,N$.
\end{proposition}

Before presenting the proof of the proposition above, for the reader's convenience, let us state Pila's theorem, see \cite{Pila}.
\begin{theorem}
\label{thm_Pila}
Let $c,\gamma\in \R$ with $c+\gamma<1$  and let $\delta>0$. Write $z=x+iy$ and suppose that $h(z)$ is holomorphic in the region $x\ge 0$, satisfying:
\begin{itemize}
	 \item[(i)]{$\limsup_{|y|\to \infty}\frac{\log|h(iy)|}{\pi |y|}\le \gamma$,}
	\item[(ii)]{$\limsup_{x\to \infty}\frac{\log|h(x)|}{2x\log x}\le c$,} 
	\item[(iii)]{$\log|h(z)|=\mathcal{O}(|z|^{2-\delta})$, throughout $x\ge 0$, as $|z|\to\infty$}. 
	\end{itemize}
Suppose that $h(m)=0$ for all $m=0,1,2,\dots$. Then $h$ vanishes identically on the set $\{z=x+iy\in\C: x\ge 0\}$. 	
\end{theorem}
\begin{remark}
\label{rem_Pila}
In Theorem \ref{thm_Pila}, it suffices to require that $h$ vanishes for all positive integers starting from some number $l \in \mathbb{N}$. In the latter case, one can apply Theorem \ref{thm_Pila} to the function $\tilde{h}(z) = z(z-1)\dots(z-(l-1))h(z)$, which satisfies growth conditions (i), (ii), and (iii), instead of the function $h$, to obtain the desired conclusion.
\end{remark}

Next, we will state a couple of lemmas required for the proof of Proposition \ref{prop_interpolate}. The following lemma is a consequence of the upper bound for the Gamma function, see \cite[formula (2.1.19) on page 34]{Paris_Kaminski_book}; see also \cite[page 300]{Olver_book},
\begin{equation}
\label{eq_bound_Gamma_f}
|\Gamma(z)| \le \sqrt{2\pi} |z|^{x-\frac{1}{2}} e^{-\frac{\pi}{2}|y|} e^{\frac{1}{6} |z|^{-1}},
\end{equation}
valid for all $ z = x + iy \in \C$, where $x\ge 0$.

\begin{lemma}\label{lem_H_bounds}
Let $\alpha \in (0,\infty)$. The function 
$$ H(z)=\Gamma(z+1+\alpha),$$
is meromorphic on $\mathbb{C}$, with its only singularities being simple poles at $z=-k-1-\alpha$ for $k=0,1,2,\dots$, and satisfies the bounds
	\begin{itemize}
		\item[(1)]{$|H(\textrm{\em i}y)|\leq C (2|y|)^{\alpha+\frac{1}{2}} e^{-\frac{\pi}{2}|y|}$ for all $y\in \R$ such that $|y|\ge 1+\alpha$,}
		\item[(2)]{$|H(x)| \leq C e^{x \log x}e^{x\log 2+(\alpha+\frac{1}{2})\log(2x)}$ for all $x\ge 1+\alpha$,} 
		\item[(3)]{$|H(z)|\leq C e^{2|z|\log(2|z|)}$ for all $z\in \{x+\textrm{\em i}y\in \C\,:\, x\ge  0, y\in \R\}$ such that $|z|\ge 1+\alpha$}.
	\end{itemize} 
Here $C:=\sqrt{2\pi}e^{\frac{1}{6}}$. 
\end{lemma}

\begin{lemma}
	\label{lem_analytic}
	Let $f \in C^{\infty}((0,\infty))$ and suppose that \eqref{f_bounds} is satisfied for some $c > 0$ and $\delta > 0$. Then the function
$$ F(z) = \int_0^\infty f(t) \, t^{-z} \, dt, \quad z \in \mathbb{C}, $$
is entirely holomorphic. Moreover, there exist constants $c_1, c_2 > 0$ depending on $c$ and $\delta$ such that:
\begin{itemize}
 \item[(i)] $|F(\mathrm{i}y)| \leq c_1$ for all $y \in \mathbb{R}$,
 \item[(ii)] $|F(x)| \leq c_2 \delta^{-x} e^{x \log x} e^{x \log 2 + \frac{1}{2} \log(2x)}$ for all $x \ge 1$,
 \item[(iii)] $|F(z)| \leq c_2 e^{2|z| \log(2|z|)}e^{|z| |\log\delta|}$ for all $z \in \{x + \mathrm{i}y \in \mathbb{C}: x \ge 0, y \in \mathbb{R}\}$ such that $|z| \ge 1$.
\end{itemize}
\end{lemma}

\begin{proof}
The fact that $F$ is an entire holomorphic function follows from \cite[Theorem 3.3.7]{Lerner}, given that the function $\C\ni z \mapsto t^{-z}$ is holomorphic when $t > 0$, and the bounds \eqref{f_bounds}.

Let us prove the bounds (i)--(iii). To that end, we first note that $t^{-z} = e^{-z \log t}$ for $t > 0$. Bound (i) follows since $f \in L^1((0,\infty))$ by \eqref{f_bounds}. For (ii) and (iii), we write for each $z = x + \mathrm{i}y$, with $x > 0$ and $y \in \mathbb{R}$,	
\begin{multline*}
		|F(z)|\leq  \int_0^1  |f(t)|\,t^{-x}\,dt+ \int_1^{\infty}|f(t)|\,t^{-x}\,dt 
		\leq  c\int_0^1 e^{-\frac{\delta}{t} } \,t^{-x}\,dt+ c\int_1^{\infty}e^{-\delta t}\,t^{-x}\,dt \\
		=  c\int_1^\infty e^{-\delta t} \,t^{x-2}\,dt+ c\int_1^{\infty}e^{-\delta t}\,t^{-x}\,dt \leq 2c  \int_1^\infty e^{-\delta t} \,t^{x}\,dt\\
		=2c\, \delta^{-x-1}  \int_\delta^\infty e^{-t} \,t^{x}\,dt \leq 2c\, \delta^{-x-1}\,\Gamma(x+1). 
	\end{multline*}
Assuming that $x \ge 1$, bound (ii) follows by combining the preceding bound with \eqref{eq_bound_Gamma_f}. Bound (iii) follows similarly.
\end{proof}

\begin{lemma}
	\label{lem_stone}
	Let $f\in C^{\infty}((0,\infty))$ and suppose that \eqref{f_bounds} is satisfied for some $c>0$ and $\delta>0$. Suppose also that 
	\begin{equation}\label{s_f_density}  \int_0^{\infty} f(t)\,t^{k}\,dt=0 \quad \text{for $k=0,1,\ldots$}.\end{equation}
	Then, $f$ vanishes identically.
	\end{lemma}

\begin{proof}
	Consider the the Fourier transform of $1_{[0,\infty)}f$,
	$$\mathcal F(1_{[0,\infty)}f)(\xi) = \int_0^{\infty}f(s)e^{-\textrm{i}\,\xi s}\,ds.$$
	Applying of \cite[Theorem 3.3.7]{Lerner}, it follows that $\mathcal F(1_{[0,\infty)}f)(\xi)$ is holomorphic for all $\xi \in \C$ that satisfy $\textrm{Im}\, \xi < \delta$. In view of \eqref{s_f_density}, we deduce that $\mathcal F(1_{[0,\infty)}f)(\xi)$ and all its derivatives vanish at the point $\xi=0$, thus yielding that $f$ must vanish identically. 
\end{proof}

\begin{proof}[Proof of Proposition~\ref{prop_interpolate}]
	Let us define the countable set $\mathbb A $ as follows
	$$ \mathbb A =\{z\in \C\,:\, -z-1-\alpha_j \in \{0\}\cup \N \quad \text{for some $j=1,\ldots,N$}\}\subset (-\infty,-1).$$
	We define the function 
	$$ h(z):=\sum_{j=1}^N H_j(z)\, F_j(z),   \quad  z\in \C,$$
where 
\[
H_j(z):=\Gamma(z+1+\alpha_j), \quad F_j(z):=\int_0^{\infty} f_j(t)\, t^{-z}\,dt, \quad j=1,\ldots,N.
\]
Considering Lemma~\ref{lem_analytic} and the fact that $H_j$ are holomorphic on $\mathbb{C} \setminus \mathbb{A}$ with simple poles at each $z=-k-1-\alpha_j$, $k=0,1,2,\dots$, we observe that the function $h$ is holomorphic on $\mathbb{C} \setminus \mathbb{A}$ and also has (at most) simple poles at each $z \in \mathbb{A}$.	
	
Applying the bounds (1)--(3) in Lemma~\ref{lem_H_bounds} for the functions $H_j(z)$ and the bounds (i)-(iii) in Lemma~\ref{lem_analytic} for $F_j(z)$ for $\textrm{Re}\, z\geq 0$ and with $j=1,\ldots,N$, we deduce that the function $h(z)$ satisfies the growth rates in the right half plane that is stated in Theorem \ref{thm_Pila} (with $c=1$ and $\gamma=-\frac{1}{2}$ in the statement of the theorem) and consequently, as $h(z)$ vanishes on positive integers greater than or equal to $l$ (see \eqref{f_density}), we must have $h(z)=0$ for all $z$ in the right half plane. 

Noting that $\mathbb C\setminus \mathbb A$ is connected, it follows from analytic continuation that
	\begin{equation}\label{h_zero}
	h(z) =0 \qquad \forall\, z\in \mathbb C \setminus \mathbb A.
	\end{equation}
	Let us now fix $j\in \{1,\ldots N\}$ and consider for each $k=0,1,2,\ldots$ the following limit,
	$$ \lim_{z\to -\alpha_j-1-k} \left(z+1+\alpha_j+k\right)\, h(z).$$
	On the one hand, the limit above is zero, owing to \eqref{h_zero}. On the other hand, using property \eqref{diff_int}, we note that
	$$  \lim_{z\to -\alpha_j-1-k} \left(z+1+\alpha_j+k\right)\, \Gamma(z+1+\alpha_l)\,\int_0^{\infty} f_l(t) \,t^{-z}\,dt=0 \quad \text{for all $l\neq j$.}
	$$
	Hence, 
	$$
	 \lim_{z\to -\alpha_j-1-k} \left(z+1+\alpha_j+k\right)\, \Gamma(z+1+\alpha_j)\,\int_0^{\infty} f_j(t) \,t^{-z}\,dt=0.
	 $$
	 Noting that Gamma functions have simple poles at nonpositive integers $-k$, we conclude via residue calculus that 
	 $$
	 \int_0^\infty f_j(t)\,t^{1+\alpha_j}\,t^{k}\,dt=0 \quad \text{for all $k=0,1,\ldots$.}
	 $$ 
	 The claim now follows from Lemma~\ref{lem_stone}.
	\end{proof}

Let us now turn to the proof of the entanglement principle, Theorem~\ref{thm_ent}. We first observe that it suffices to establish Theorem~\ref{thm_ent} in the case when all $\alpha_j \in (0,1)$ for $j = 1, \dots, N$. Specifically, Theorem~\ref{thm_ent} follows from the following result and the unique continuation principle for the Laplacian.  

\begin{lemma}
	\label{lem_ent_new}
	 Let $\{v_j\}_{j=1}^N\subset  C^{\infty}(M)$, where $N\in \N$,  satisfy \eqref{v_12_O} for some $b_1,\ldots,b_N \in \C\setminus\{0\}$, and some $\{\alpha_j\}_{j=1}^N \subset (0,1)$ such that  $\alpha_j \ne \alpha_k$  for all $j,k=1,\ldots,N$ with $j\neq k$. Then $v_j=0$ on $M$ for all $j=1,\ldots,N$.
\end{lemma} 

\begin{proof}[Proof of Theorem~\ref{thm_ent}]
To see that Theorem~\ref{thm_ent} indeed follows from Lemma \ref{lem_ent_new}, we proceed as follows. If some $\alpha_j > 1$ in Theorem~\ref{thm_ent}, we express it as $\alpha_j = m_j + \tilde{\alpha}_j$, where $m_j \in \mathbb{N}$ is the integer part of $\alpha_j$ and $\tilde{\alpha}_j \in (0,1)$. Then, by the functional calculus, see \cite[Theorem 4.15]{Dimassi_Sjostrand_book},
 we have
\[
(-\Delta_g)^{\alpha_j}v_j = (-\Delta_g)^{\tilde \alpha_j}(-\Delta_g)^{m_j}v_j.
\]
Let $\tilde{v}_j := (-\Delta_g)^{m_j} v_j$ for this $j$. We set $\tilde{\alpha}_j := \alpha_j$ and $\tilde{v}_j := v_j$ if $\alpha_j \in (0,1)$ in Theorem~\ref{thm_ent}. Note that the condition \eqref{diff_int} implies that $\tilde \alpha_j \neq \tilde \alpha_k$ for all $j,k=1,\ldots,N$ with $j \neq k$. Applying Lemma \ref{lem_ent_new} with $\tilde{\alpha}_j$ and $\tilde{v}_j$ allows us to conclude that $\tilde{v}_j = 0$ on $M$. If $\tilde{v}_j = (-\Delta_g)^{m_j} v_j$, using the fact that $v_j|_{\mathcal{O}} = 0$ and the unique continuation principle for the Laplacian, we obtain $v_j = 0$ on $M$, thereby establishing Theorem~\ref{thm_ent}.
\end{proof}

%Finally, note in view of the spectral resolution of $-\Delta_{g}$ on $(M,g)$ that there holds
%\begin{equation}
%	\label{heat_kernel_bound_2}
%	\|e^{t\Delta_{g}}\|_{L^2(M)\to L^{2}(M)} \leq e^{-\beta t} \quad t>0,
%\end{equation}
%for some $\beta>0$ depending only on $(M,g)$.

\begin{proof}[Proof of Lemma \ref{lem_ent_new}]
First, let $m=1,2,\dots$, and observe that the equality \eqref{v_12_O} implies that
\begin{equation}
\label{eq_100_1}
(\Delta_g^{m+1}v_1)|_{\mathcal{O}}=\dots=(\Delta_g^{m+1}v_N)|_{\mathcal{O}}=0,
\end{equation}
and 
\begin{equation}
\label{eq_100_2}
\sum_{j=1}^N b_j((-\Delta_g)^{\alpha_j}\Delta_g^{m+1} v_j)|_{\mathcal{O}}=0.
\end{equation}
Note that the functional calculus of self-adjoint operators was employed to get~\eqref{eq_100_2}. 

Let $\omega \subset\subset \mathcal{O}$ be a nonempty open set. Due to the fact that $\alpha_j \in (0,1)$, by utilizing \eqref{frac_laplace_heat_pointwise} along with \eqref{eq_100_1}, we derive from \eqref{eq_100_2} that
\begin{equation}
\label{eq_100_3}
\sum_{j=1}^N b_j \frac{1}{\Gamma(-\alpha_j)}\int_0^\infty  (e^{t\Delta_g} \Delta_g^{m+1} v_j)(x) \frac{dt}{t^{1+\alpha_j}}=0,
\end{equation}
for $x\in \omega$ and $m=1,2,\dots$.  

Next, we will argue as in \cite[proof of Proposition 3.1]{Ghosh_Uhlmann_2021}, see also \cite[proof of Theorem 1.1]{FGKU}. Using the fact that the function $t\mapsto e^{t\Delta_g} (\Delta_g v_j)\in C^\infty([0,\infty); C^\infty(M))$, and  that $e^{t\Delta_g}\Delta_{g}^{m}= \Delta_{g}^{m} e^{t\Delta_g}$ for all $t\ge 0$ on $\mathcal{D}(\Delta_{g}^{m})$, we obtain for any $m=1,2,\dots$,
\begin{equation}
\label{eq_100_4}
\big(e^{t\Delta_g}\Delta_g^{m} (\Delta_g v_j)\big)(x)=\p_t^{m} \big(e^{t\Delta_g} (\Delta_g v_j)\big)(x), 
\end{equation}
for  $x\in \omega$ and $j=1,\dots, N$. Combining \eqref{eq_100_3} and  \eqref{eq_100_4}, we get 
\begin{equation}
\label{eq_100_5}
\sum_{j=1}^N b_j \frac{1}{\Gamma(-\alpha_j)}\int_0^\infty  \p_t^m(e^{t\Delta_g} \Delta_g v_j)(x) \frac{dt}{t^{1+\alpha_j}}=0,
\end{equation}
for $x\in \omega$ and $m=1,2,\dots$.  

We shall repeatedly integrate by parts in \eqref{eq_100_5} $m$ times. We claim that there will be no contributions from the endpoints. Indeed, fixing $s > n/2$ such that $s/2 \in \mathbb{N}$, for $t > 0$ and $x \in \omega$, using \eqref{eq_100_4} and \eqref{eq_100_-1}, we obtain 
\begin{equation}
\label{eq_100_5_1}
|\p_t^l (e^{t\Delta_g}\Delta_g v_j)(x)|=|e^{t\Delta_g}\Delta_g^{l+1}v_j(x)|\le Ce^{-\beta t}\|(1-\Delta_{g})^{s/2}\Delta_g^{l+1} v_j\|_{L^2(M)},
\end{equation}
where $\beta > 0$ and $l = 0, 1, \dots, m-1$. The bound \eqref{eq_100_5_1} shows that no contribution from $t = \infty$ arises when integrating by parts in \eqref{eq_100_5}.

To demonstrate that no contribution arises at $t=0$ when integrating by parts in \eqref{eq_100_5}, using \eqref{eq_100_4}, we first obtain, for $t>0$ and $x \in \omega$,
\begin{equation}
\label{eq_100_6}
\p_t^l (e^{t\Delta_g}\Delta_g v_j)(x)=\int_{M\setminus \mathcal{O}} e^{t\Delta_g}(x,y) (\Delta_g^{l+1} v_j)(y)dV_g(y),
\end{equation}
for $l=0, 1, \ldots, m-1$. Here $dV_g$ is the Riemannian volume element. Using \eqref{Gauss_heat}, we derive from \eqref{eq_100_6} that for $0 < t < 1$ and $x \in \omega$,
\begin{equation}
\label{eq_100_7}
|\p_t^l (e^{t\Delta_g}\Delta_g v_j)(x)|\le \|e^{t\Delta_g}(\cdot,\cdot)\|_{L^\infty(\omega\times (M\setminus\mathcal{O}))}\|\Delta_g^{l+1} v_j\|_{L^1(M)}\le Ce^{-\frac{\tilde c}{t}}\|\Delta_g^{l+1} v_j\|_{L^1(M)},
\end{equation}
$l=0, 1, \ldots, m-1$, and $j=1, \ldots, N$. Here $\tilde{c} > 0$ depends on $\mathrm{dist}_{g}(\overline{\omega}, M \setminus \mathcal{O}) > 0$. The bound \eqref{eq_100_7} shows that no contribution arises at $t=0$ when integrating by parts in \eqref{eq_100_5}.

Integrating by parts $m$ times in \eqref{eq_100_5}, we obtain that 
\begin{equation}
\label{eq_100_8}
\sum_{j=1}^N b_j \frac{1}{\Gamma(-\alpha_j)}\kappa_j\int_0^\infty  (e^{t\Delta_g} \Delta_g v_j)(x) \frac{dt}{t^{m+1+\alpha_j}}=0,
\end{equation}
for $x\in \omega$ and $m=1,2,\dots$.  Here 
\begin{equation}
\label{eq_100_9}
\kappa_j=(-1)^m(1+\alpha_j)(2+\alpha_j)\dots (m+\alpha_j)=(-1)^m\frac{\Gamma(m+1+\alpha_j)}{\Gamma(1+\alpha_j)}.
\end{equation}
Fixing an arbitrary point $x\in \omega$ and letting 
\begin{equation}
\label{eq_100_10}
f_j(t)=\frac{b_j}{\Gamma(-\alpha_j)\Gamma(1+\alpha_j)}(e^{t\Delta_g} \Delta_g v_j)(x) t^{-(1+\alpha_j)}= 
\frac{-b_j \sin(\pi \alpha_j)}{\pi}(e^{t\Delta_g} \Delta_g v_j)(x) t^{-(1+\alpha_j)},
\end{equation}
for $t>0$, $j=1,\dots, N$, 
we get from \eqref{eq_100_8}, \eqref{eq_100_9}, and \eqref{eq_100_10},  that 
\begin{equation}
\label{eq_100_11}
\sum_{j=1}^N \Gamma(m+1+\alpha_j)\int_0^\infty f_j(t) t^{-m}dt=0, \quad m=1,2,\dots.
\end{equation}
Note that in the second equality in \eqref{eq_100_10}, we used the reflection formula for the Gamma function
\[
\Gamma(-\alpha_j)\Gamma(1+\alpha_j)=-\frac{\pi}{\sin(\pi \alpha_j)},
\]
see \cite[formula (2.1.20) on page 35]{Paris_Kaminski_book}. We have that $f_j\in C^\infty((0,\infty))$, $j=1,\dots, N$. Using the bounds \eqref{eq_100_5_1}  and  \eqref{eq_100_7} with $l=0$, we get for $t\ge 1$,  	
\begin{equation}
\label{eq_100_12}
|f_j(t)|\le C|b_j| e^{-\beta t}\|(1-\Delta_{g})^{s/2}\Delta_g  v_j\|_{L^2(M)}\le Ce^{-\beta t},
\end{equation}	
and for $t\in (0,1]$,
\begin{equation}
\label{eq_100_13}
|f_j(t)|\le C|b_j| e^{-\frac{\tilde c}{t}}\|\Delta_g v_j\|_{L^1(M)}t^{-2}\le Ce^{-\frac{\tilde c}{2t}},
\end{equation}	
for all $j=1,\dots, N$. The bounds \eqref{eq_100_12} and \eqref{eq_100_13} show that the functions $f_j$, $j=1,\ldots,N$ , satisfy the bounds \eqref{f_bounds}. In view of \eqref{eq_100_11}, we conclude, by applying Proposition~\ref{prop_interpolate}, that $f_j$ must vanish identically for  $j=1,\ldots,N$. Hence, 
\[
 (e^{t\Delta_g}\Delta_g v_j)(x)=0, \quad t>0,\quad  x\in \omega, \quad j=1,\ldots,N.
 \]
Furthermore, the function $e^{t\Delta_g}\Delta_g v_j \in C^\infty((0,\infty)\times M)$ satisfies the heat equation $(\partial_t - \Delta_g)(e^{t\Delta_g}\Delta_g v_j) = 0$ in $M$, $j = 1, \dots, N$. It follows from the unique continuation principle for heat equations (see, e.g., \cite[Sections 1 and 4]{Lin}) together with the fact that $M$ is connected that $(e^{t\Delta_g}\Delta_g v_j)(x) = 0$ for all $t > 0$ and $x \in M$, $j = 1, \ldots, N$. Using that $\lim_{t\to 0}(e^{t\Delta_g}\Delta_g v_j)(x)=(\Delta_g v_j)(x)$ for any $x\in M$, see \cite[Theorem 4.1.4, page 105]{Hsu_book}, we get
$$\Delta_g v_1 = \ldots = \Delta_g v_N = 0 \quad \text{on } M.$$
Finally, in view of \eqref{v_12_O} and the unique continuation principle for elliptic equations, we deduce that $v_1 = \ldots = v_N = 0$ on $M$.
 \end{proof}

\begin{remark}
\label{rem_UCP_N_1}
The case $N=1$ in Theorem~\ref{thm_ent} and Lemma~\ref{lem_ent_new} corresponds to the unique continuation principle for the fractional Laplacian $(- \Delta_g)^\alpha$, with $\alpha \in (0,\infty) \setminus \N$, on a closed connected manifold $(M,g)$. In this case, the proof does not require the use of Theorem~\ref{thm_Pila} by Pila and simply follows from Lemma~\ref{lem_stone} in view of \eqref{eq_100_11} with $N=1$. We also refer the reader to \cite{Ghosh_Uhlmann_2021} for the proof of unique continuation principle for fractional Laplacian on $\R^n$ using heat semigroups. 
\end{remark}

Next, for future reference, we shall state some consequences of Theorem~\ref{thm_ent}. We refer to Proposition~\ref{prop_direct} and Remark~\ref{rem_direct} for the notation used in the results stated below.
\begin{lemma}
\label{lem_analog_5_2}
Let $\alpha \in (0,1)$. Let $(M, g)$ be a smooth closed and connected Riemannian manifold, and let $V \in C^{\infty}(M)$.  Assume that $\dim(\mathcal{K}_{M,g,V}) = N$, where $1 \le N < \infty$. Let $\zeta_1, \dots, \zeta_N \in \mathcal{K}_{M,g,\overline{V}}$ be linearly independent functions. Then, for any $c = (c_1, \dots, c_N) \in \mathbb{C}^N$, there exists a function $h \in C^\infty_0(O)$ such that $(h, \zeta_l)_{L^2(O)} = c_l$ for $l = 1, \dots, N$.
\end{lemma}
\begin{proof}
We shall show that the following linear map
\[
L: C_0^\infty(O) \ni h \mapsto \big((h, \zeta_1)_{L^2(O)}, \dots, (h, \zeta_N)_{L^2(O)}\big) \in \mathbb{C}^N
\]
is surjective. Suppose, for the sake of contradiction, that $L$ is not surjective. Then there exists a nonzero vector $b = (b_1, \dots, b_N) \in \mathbb{C}^N$ such that 
\begin{equation}
\label{eq_lemma_anal_1}
0 = L(h) \cdot b = (h, \zeta)_{L^2(O)},
\end{equation}
for all $h \in C^\infty_0(O)$. Here, $\zeta := \sum_{l=1}^N b_l \zeta_l \in C^\infty(M)$, and $v \cdot w = \sum_{l=1}^N v_l \overline{w}_l$ denotes the inner product of vectors $v, w \in \mathbb{C}^N$. It follows from \eqref{eq_lemma_anal_1} that $\zeta|_{O} = 0$. Since $\zeta \in \mathcal{K}_{M,g,\overline{V}}$, we have $((-\Delta_g)^{\alpha} \zeta)|_{O} = 0$. By the unique continuation principle for $(-\Delta_g)^{\alpha}$ of Theorem \ref{thm_ent} with $N = 1$, we conclude that $\zeta = 0$ on $M$, implying that $b = 0$, which is a contradiction.
\end{proof}

We refer to Definition \ref{def_S} for the notation used in the statement of the following lemma. This lemma will be used in the next section and is a key ingredient in the reduction step from the Cauchy data set $\mathcal C^O_{M,g,V}$ to the variant of Gel'fand inverse spectral problem that was discussed in the introduction.
\begin{lemma}
	\label{density_lemma}
	Let $\alpha \in (0,1)$. Let $(M, g)$ be a smooth closed and connected Riemannian manifold, and let $V \in C^{\infty}(M)$ satisfy $V|_{O} = 0$. Let $\phi \in C^{\infty}(M)$ be an eigenfunction of $-\Delta_g$ on $(M, g)$ corresponding to some eigenvalue $\lambda > 0$. Then, there holds,
	\begin{equation}
	\label{S_ortho_100}
	(S_{M,g,V}(f),\phi)_{L^2(M)} \neq 0 \quad \text{for some $f\in \mathcal H_{M,g,V}^O$}.
	\end{equation}
\end{lemma}

\begin{proof}
\noindent{\bf Case I}. Assume that $\mathcal{K}_{M,g,V} = \{0\}$. To show \eqref{S_ortho_100}, we shall argue by contradiction. We assume that
\begin{equation}
\label{S_ortho_trivial_case}
(S_{M,g,V}(f), \phi)_{L^2(M)} = 0 \quad \forall f \in C^\infty_0(O).
\end{equation}
By Proposition \ref{prop_direct} and Remark \ref{rem_direct}, the equation
\begin{equation}
\label{S_ortho_trivial_case_2}
(-\Delta_g)^\alpha v + \overline{V} v = \phi \quad \text{on} \quad M
\end{equation}
has a unique solution $v \in C^\infty(M)$. It follows from \eqref{S_ortho_trivial_case} and \eqref{S_ortho_trivial_case_2} that
\[
0=(S_{M,g,V}(f),\phi)_{L^2(M)}=(S_{M,g,V}(f),(-\Delta_g)^\alpha v+\overline{V}v)_{L^2(M)} = (f,v)_{L^2(M)},
\]
for all $f \in C^\infty_0(O)$. This shows that $v|_{O} = 0$.

Applying $-\Delta_g$ to \eqref{S_ortho_trivial_case_2}, we obtain
 \begin{equation}
	\label{S_ortho_trivial_case_3} 
(-\Delta_g)^\alpha (-\Delta_g) v+(-\Delta_g)(\overline{V}v)=\lambda \phi \quad  \text{on} \quad M, 
\end{equation} 
and multiplying \eqref{S_ortho_trivial_case_2} by $-\lambda$ yields
\begin{equation}
	\label{S_ortho_trivial_case_4} 
-\lambda(-\Delta_g)^\alpha v-\lambda \overline{V}v=-\lambda\phi \quad  \text{on} \quad M.
\end{equation}
Adding \eqref{S_ortho_trivial_case_3} and \eqref{S_ortho_trivial_case_4}, we obtain
\begin{equation}
	\label{S_ortho_trivial_case_5} 
(-\Delta_g)^\alpha v_1=(-\Delta_g-\lambda)(\overline{V}v)  \quad  \text{on} \quad M, 
\end{equation}  
where
\begin{equation}
	\label{S_ortho_trivial_case_6} 
v_1=(\Delta_g+\lambda)v.
\end{equation} 
Since $v|_{O} = 0$, we conclude from \eqref{S_ortho_trivial_case_5} and \eqref{S_ortho_trivial_case_6} that $((-\Delta_g)^\alpha v_1)|_{\mathcal{O}} = 0$ and $v_1|_{O} = 0$. By Theorem \ref{thm_ent} with $N = 1$, thanks to the fact that $M$ is connected, we conclude that $v_1 = 0$ on $M$. In view of \eqref{S_ortho_trivial_case_6}, we get $(\Delta_g + \lambda)v = 0$ on $M$. As $v|_{O} = 0$, by the elliptic unique continuation, we conclude that $v = 0$ on $M$. Therefore, it follows from \eqref{S_ortho_trivial_case_2} that $\phi = 0$ on $M$, which is a contradiction.

\noindent{\bf Case II}. Assume that $\dim(\mathcal{K}_{M,g,V}) = N$, where $1 \leq N < \infty$. First, we claim that $\phi \notin \mathcal{K}_{M,g,V}$. Indeed, if $\phi \in \mathcal{K}_{M,g,V}$, then $\lambda^\alpha \phi + V \phi = 0$ in $M$. Given that $V|_{O} = 0$ and $\lambda \neq 0$, it follows that $\phi|_{O} = 0$. By the unique continuation property, $\phi = 0$ in $M$, which is a contradiction. Thus, the claim follows.

Now, to prove \eqref{S_ortho_100}, we proceed similarly to Case I and argue by contradiction. We assume that
\begin{equation}
\label{S_ortho}
(S_{M,g,V}(f), \phi)_{L^2(M)} = 0 \quad \forall f \in \mathcal{H}_{M,g,V}^O.
\end{equation}
Let $\{\eta_k\}_{k=1}^N \subset \mathcal{K}_{M,g,V}$ be an $L^2(M)$-orthonormal basis for $\mathcal{K}_{M,g,V}$. Let
$$\Pi: L^2(M) \to \mathcal{K}_{M,g,V}$$
be the projection operator defined by
$$\Pi(u) = \sum_{k=1}^N (u, \eta_k)_{L^2(M)} \eta_k \quad \forall u \in L^2(M).$$
Using the definition of $S_{M,g,V}(f)$ (see Definition~\ref{def_S}) and \eqref{S_ortho}, we deduce that
\begin{equation}
\label{phi_proj}
(S_{M,g,V}(f), \phi - \Pi(\phi))_{L^2(M)} = 0 \quad \forall f \in \mathcal{H}_{M,g,V}^O.
\end{equation}

Next, using the fact that
$$(\phi - \Pi(\phi), v) = 0 \quad \forall v \in \mathcal{K}_{M,g,V},$$
we may use Proposition \ref{prop_direct} and Remark \ref{rem_direct} to deduce that there exists some solution $\tilde{w} \in C^{\infty}(M)$ to the equation
\begin{equation}
\label{lem_density_new_0}
(-\Delta_g)^\alpha \tilde{w} + \overline{V} \tilde{w} = \phi - \Pi(\phi) \quad \text{on } M.
\end{equation}
Using this equation together with \eqref{phi_proj}, we deduce that
\begin{equation}
\label{lem_density_new_1}
0 = (S_{M,g,V}(f), \phi - \Pi(\phi))_{L^2(M)} = (S_{M,g,V}(f), (-\Delta_g)^\alpha \tilde{w} + \overline{V} \tilde{w})_{L^2(M)} = (f, \tilde{w})_{L^2(O)},
\end{equation}
for all $f \in \mathcal{H}_{M,g,V}^O$.
	 
Let $\mathcal{K}_{M,g,\overline{V}} = \text{Span}\{\zeta_1, \dots, \zeta_N\}$, where $\zeta_1, \dots, \zeta_N \in C^\infty(M)$ is a collection of linearly independent functions on $M$. By the unique continuation principle for $(-\Delta_g)^{\alpha}$ (see Theorem \ref{thm_ent} with $N = 1$), we have that $\zeta_1|_{O}, \dots, \zeta_N|_{O}$ are also linearly independent on $O$. Let us define
\[
W = \text{Span}\{\zeta_1|_{O}, \dots, \zeta_N|_{O}\} \subset L^2(O).
\]
Writing $L^2(O) = W \oplus W^\perp$, we deduce that
\begin{equation}
\label{lem_density_new_1_2}
\tilde{w}|_{O} = \zeta + w_0,
\end{equation}
where $\zeta \in W$ and $w_0 \in W^\perp$, i.e.,
\begin{equation}
\label{lem_density_new_2}
(w_0, \zeta_k|_{O})_{L^2(O)} = 0 \quad \text{for all} \quad k = 1, \dots, N.
\end{equation}
We shall next show that the condition \eqref{lem_density_new_1} implies that $w_0 = 0$. To that end, let $\{h_\ell\}_{\ell=1}^\infty \subset C_0^\infty(O)$ be such that
\begin{equation}
\label{lem_density_new_3}
\|h_\ell - w_0\|_{L^2(O)} \to 0 \quad \text{as} \quad \ell \to \infty.
\end{equation}
It follows from \eqref{lem_density_new_2} and \eqref{lem_density_new_3} that
\begin{equation}
\label{lem_density_new_4}
\lim_{\ell \to \infty} (h_\ell, \zeta_k)_{L^2(O)} = 0 \quad \text{for all} \quad k = 1, \dots, N.
\end{equation}
By Lemma \ref{lem_analog_5_2}, there exist functions $\{\theta_k\}_{k=1}^N \subset C_0^\infty(O)$ such that
\begin{equation}
\label{lem_density_new_5}
(\theta_k, \zeta_j) = \delta_{kj} \quad \text{for all} \quad k, j = 1, \dots, N.
\end{equation}
Consider the sequence of functions $f_\ell \in C_0^\infty(O)$ defined by
\[
f_\ell = h_\ell - \sum_{j=1}^N (h_\ell, \zeta_j)_{L^2(O)} \theta_j, \quad \ell = 1,2, \dots.
\]
It follows from \eqref{lem_density_new_5} that $(f_\ell, \zeta_k)_{L^2(O)} = 0$ for all $k = 1, \dots, N$, $\ell=1,2,\dots$, and therefore, $f_\ell \in \mathcal{H}_{M,g,V}^O$ for all $\ell = 1, 2, \dots$. Thus, we conclude from \eqref{lem_density_new_1} and \eqref{lem_density_new_1_2} that
\begin{equation}
\label{lem_density_new_6}
(f_\ell, w_0)_{L^2(O)} = 0 \quad \text{for all} \quad \ell = 1, 2, \dots.
\end{equation}
We observe from \eqref{lem_density_new_3} and \eqref{lem_density_new_4} that
\begin{equation}
\label{lem_density_new_7}
\|f_\ell - w_0\|_{L^2(O)} \to 0 \quad \text{as} \quad \ell \to \infty.
\end{equation}
It follows from \eqref{lem_density_new_6} and \eqref{lem_density_new_7} that $(w_0, w_0)_{L^2(O)} = 0$ and therefore, $w_0 = 0$, proving the claim.

Thus, in view of \eqref{lem_density_new_1_2}, we see that $\tilde{w}|_{O} = \zeta$ where $\zeta \in W$. Letting $\tilde{\zeta} \in \mathcal{K}_{M,g,\overline{V}}$ be such that $\tilde{\zeta}|_{O} = \zeta$, and defining $w = \tilde{w} - \tilde{\zeta}$, we deduce from \eqref{lem_density_new_0} that the function $w$ satisfies
\begin{equation}
\label{def_w}
(-\Delta_g)^\alpha w + \overline{V} w = \phi - \Pi(\phi) \quad \text{on } M \quad \text{and} \quad w|_{O} = 0.
\end{equation}

We claim that the above equation implies that $\phi=0$ on $M$, which would yield a contradiction. To this end, we recall that $w, \phi \in C^{\infty}(M)$ and apply $(-\Delta_g)^\alpha +V$ to both sides of the above equation to obtain
\begin{equation}\label{w_1}
((-\Delta_g)^\alpha +V)((-\Delta_g)^\alpha + \overline V) w = ((-\Delta_g)^\alpha +V) \phi \ne 0 \quad \text{on $M$} \quad \text{and} \quad w|_{O}=0.
\end{equation}
Here we used that $\phi \notin \mathcal K_{M,g,V}$. Rearranging the above equation, we write
\begin{equation}\label{w_2}
(-\Delta_g)^{2\alpha} w + (-\Delta_g)^\alpha (\overline V w) - (-\Delta_g)^\alpha \phi = F,
\end{equation}
where
$$
F = V\phi - V\overline V w - V (-\Delta_g)^\alpha w.
$$
In particular, since $V=0$ on the set $O$, we deduce that the function $F$ vanishes on the set $O$. Therefore, 
  \begin{equation}\label{w_3}
  	(-\Delta_g)^{2\alpha} w + (-\Delta_g)^\alpha(\overline V w)- \lambda^\alpha \phi = 0 \quad \text{on $O$.}
  \end{equation}
Next, by applying $-\Delta_g$ to the equation \eqref{w_2} and noting again that $\Delta_g F$ also vanishes on the set $O$, we arrive at the equation 
 \begin{equation}\label{w_4}
	(-\Delta_g)^{1+2\alpha} w + (-\Delta_g)^{1+\alpha}(\overline V w)- \lambda^{1+\alpha} \phi = 0 \quad \text{on $O$.}
	\end{equation}
Finally, multiplying equation \eqref{w_3} with $-\lambda$ and adding it to \eqref{w_4} we obtain 
 \begin{equation}\label{w_5}
	(-\Delta_g)^{2\alpha} v_1 = (-\Delta_g)^{\alpha}v_2 \quad \text{on $O$,}
\end{equation}
where $v_1,v_2 \in C^{\infty}(M)$ are defined by
\begin{equation}
	\label{v_12_exp}
	v_1 = \Delta_g w +\lambda w \quad \text{and} \quad v_2= -\Delta_g(\overline V w) - \lambda \overline V w \quad \text{on $M$}.
\end{equation}
Note that as $w=0$ on the set $O$, we have $v_1=v_2=0$ on the set $O$ as well.

Let us recall that $\alpha \in (0,1)$ and consider two cases: $2\alpha = 1$ or $2\alpha \in (0,1) \cup (1,2)$.

\noindent{\bf Case II.1.} We assume that $\alpha \in (0,1/2) \cup (1/2,1)$. As $v_1|_{O} = v_2|_{O} = 0$, equation \eqref{w_5} implies that 
$$ 
v_1 = v_2 = 0 \quad \text{on $M$},
$$
thanks to Theorem~\ref{thm_ent}. Next, using the definition of $v_1$ in \eqref{v_12_exp}, the fact that $w|_{O} = 0$, together with the elliptic unique continuation principle, we deduce that 
$w$ must vanish identically on $M$. As $\lambda > 0$, equation \eqref{w_3} now implies that $\phi|_{O}$ must be zero, and subsequently by the unique continuation principle for elliptic equations, $\phi = 0$ on $M$, which yields a contradiction to it being an eigenfunction.

\noindent{\bf Case II.1.} We assume the remaining case $\alpha=1/2$. Then, it follows from \eqref{w_5} together with the fact that $v_1|_{O}=0$ that 
$$ 
(-\Delta_g)^{1/2} v_2 = 0 \quad \text{on $O$}, \quad \text{and} \quad v_2|_{O} = 0.
$$
From the unique continuation property for fractional operators, we have that $v_2 = 0$ everywhere on $M$, see Theorem~\ref{thm_ent}. Using the definition of $v_2$ together with unique continuation principles for elliptic equations, we deduce now that
\begin{equation}\label{Vbarw}
\overline{V} w = 0 \quad \text{on $M$}.
\end{equation}
Note that as $2\alpha = 1$ and as $w$ vanishes on $O$, the function $(-\Delta_g)^{2\alpha} w$ must also vanish on the set $O$. Returning to \eqref{w_3}, using the previous observation together with \eqref{Vbarw}, as well as the fact that $\lambda > 0$, we conclude that $\phi = 0$ on $O$. Using the unique continuation principle for elliptic equations, it follows that $\phi = 0$ on $M$, reaching a contradiction again.
\end{proof}

\begin{remark}
Note that when $\mathcal K_{M,g,V}= \{0\}$, the condition that $V|_{O}=0$ is not needed.
\end{remark}

\section{General reduction to an inverse spectral problem}
\label{sec_reduction_to_spectral_data}
The main goal of this section is to prove that the Cauchy data set $\mathcal C_{M,g,V}^O$ uniquely determines the eigenvalues and an $L^2(M)$ Schauder basis consisting of eigenfunctions of $-\Delta_g$ on $(M,g)$. We do not need to impose the assumption (H) on the set $O$ to achieve this reduction to an inverse spectral problem. 

\begin{proposition}
	\label{prop1}
	Let $\alpha \in (0,1)$. For $j=1,2,$ let $(M_j,g_j)$ be a smooth closed and connected Riemannian manifold and let $V_j \in C^{\infty}(M_j).$ Let $O\subset M_1\cap M_2$ be a nonempty connected open set such that $M_j\setminus\overline{O}$ is nonempty, and assume that $g_1|_{O}=g_2|_{O}$ and  $V_1|_{O}=V_2|_O=0.$ Suppose that
	$$ \mathcal C_{M_1,g_1,V_1}^{O} = \mathcal C_{M_2,g_2,V_2}^{O}.$$
	Then, for $j=1,2,$ there exists an $L^2(M_j)$ Schauder basis consisting of eigenfunctions $\{\psi_k^{(j)}\}_{k=0}^{\infty}\subset C^{\infty}(M_j)$ for $-\Delta_{g_j}$ on $(M_j,g_j)$ corresponding to (not necessarily distinct) eigenvalues 
	$$0=\mu_0^{(j)}<\mu_1^{(j)}\leq \mu_2^{(j)}\leq \mu_3^{(j)} \ldots$$
	such that given any $k=0,1,2,\ldots,$ there holds
\begin{equation}\label{Gelfand_claim} \mu_k^{(1)}= \mu_k^{(2)} \quad \text{and} \quad \psi_k^{(1)}(x)=\psi^{(2)}_k(x)\quad \forall \, x\in O. \end{equation} 
\end{proposition}

For the remainder of this section, we will assume that the hypotheses of Proposition~\ref{prop1} are satisfied and proceed to prove it via a series of lemmas. In the statement of the following lemma, we make use of this notation 
$$ \mathcal K_{M,g,V}^{O}=\{ v|_{O}\,:\, v\in \mathcal K_{M,g,V}\}.$$ 
We state the following lemma that follows trivially from the equality of the Cauchy data sets $\mathcal C_{M_1,g_1,V_1}^O=\mathcal C_{M_2,g_2,V_2}^O$.
\begin{lemma}
	\label{lem_K}
There holds,
\begin{equation}\label{K_1} \mathcal K_{M_1,g_1,V_1}^{O} =  \mathcal K_{M_2,g_2,V_2}^{O}.
\end{equation}
\end{lemma}

Furthermore, using $(\textrm{UCP})'$, see Theorem~\ref{thm_ent} with $N=1$, we conclude from the equality of the Cauchy data sets $\mathcal{C}_{M_1,g_1,V_1}^O = \mathcal{C}_{M_2,g_2,V_2}^O$ and $V_1|_{O} = V_2|_{O}$ that either 
\begin{equation}
\label{eq_H_O-new_1_1} 
\mathcal{K}_{M_1, g_1, V_1} = \mathcal{K}_{M_2, g_2, V_2} = \{0\} 
\end{equation}
or 
\begin{equation}
\label{eq_H_O-new_1_2}
\dim(\mathcal{K}_{M_1, g_1, V_1}) = \dim(\mathcal{K}_{M_2, g_2, V_2}) = N,
\end{equation}
for some $1 \le N < \infty$. 

Now, recalling Definition \ref{def_S}, we note that 
\begin{equation}
\label{eq_H_O-new_1}
\mathcal{H}_{M_1, g_1, V_1}^O = \mathcal{H}_{M_2, g_2, V_2}^O,
\end{equation} 
when \eqref{eq_H_O-new_1_1} holds. In view of Lemma~\ref{lem_K} together with $\mathcal{K}_{M, g, \overline{V}} = \overline{\mathcal{K}_{M, g, V}}$, we note that given any $f \in C^{\infty}_0(O)$, there holds
\begin{equation}
\label{K2}
(f, v)_{L^2(O)} = 0 \quad \forall\, v \in \mathcal{K}_{M_1, g_1, \overline{V_1}} \quad \iff \quad (f, v)_{L^2(O)} = 0 \quad \forall\, v \in \mathcal{K}_{M_2, g_2, \overline{V_2}}.
\end{equation}
This shows that \eqref{eq_H_O-new_1} also holds when \eqref{eq_H_O-new_1_2} takes place.

Let $f \in \mathcal H_{M_1,g_1,V_1}^O$ be arbitrary and consider two functions $$u_1=S_{M_1,g_1,V_1}(f) \quad \text{and} \quad u_2=S_{M_2,g_2,V_2}(f).$$ 
Note that for $j=1,2,$
\begin{equation}\label{u_1} 
	(-\Delta_{g_j})^\alpha u_j + V_j u_j = f \quad \text{on $M_j$}.
\end{equation}
As $\mathcal C^O_{M_1,g_1,V_1}= \mathcal C^O_{M_2,g_2,V_2}$ and $V_1|_{O}=V_2|_{O}$, we deduce that there exists some $\tilde u_1 \in C^{\infty}(M_1)$ and $\tilde u_2\in C^{\infty}(M_2)$ satisfying
\begin{equation} \label{u_2}
	(-\Delta_{g_j})^\alpha \tilde u_j + V_j \tilde u_j = f \quad \text{on $M_j$} \quad j=1,2,
	\end{equation}
with the additional property that
\begin{equation}
	\label{u_12}
	u_1(x)= \tilde{u}_2(x) \quad \text{and}\quad \tilde{u}_1(x)=u_2(x)\quad  \forall\, x\in O.
\end{equation}

Note that when \eqref{eq_H_O-new_1_1} holds, we have $u_1 = \tilde{u}_1$ on $M_1$ and $u_2 = \tilde{u}_2$ on $M_2$. In what follows, we shall proceed by writing four solutions $u_1$, $\tilde{u}_1$, $u_2$, and $\tilde{u}_2$ to treat both cases \eqref{eq_H_O-new_1_1} and \eqref{eq_H_O-new_1_2} simultaneously, keeping in mind the equality of solutions in the case of \eqref{eq_H_O-new_1_1}.

We remark that in the following lemma, the assumption that $V_1|_O = V_2|_O = 0$ is crucial.
\begin{lemma} 
	\label{lem_reduce_to_heat}
	Assume that $V_1|_{O} = V_2|_{O} = 0$. Let $f \in \mathcal H_{M_1,g_1,V_1}^O = \mathcal H^O_{M_2,g_2,V_2}$ be arbitrary, and let $u_1 = S_{M_1,g_1,V_1}(f) \in C^\infty(M_1)$ and $u_2 = S_{M_2,g_2,V_2}(f) \in C^\infty(M_2)$. Let $\tilde{u}_1 \in C^\infty(M_1)$ and $\tilde{u}_2 \in C^\infty(M_2)$ be given by \eqref{u_2} and \eqref{u_12}. Then there holds,
$$
(e^{t\Delta_{g_1}}(-\Delta_{g_1})^\alpha u_1)(x) = (e^{t\Delta_{g_2}}(-\Delta_{g_2})^\alpha \tilde{u}_2)(x) \quad \forall\, t \in (0,\infty) \quad \forall\, x \in O,
$$
and
$$
(e^{t\Delta_{g_1}}(-\Delta_{g_1})^\alpha \tilde{u}_1)(x) = (e^{t\Delta_{g_2}}(-\Delta_{g_2})^\alpha u_2)(x) \quad \forall\, t \in (0,\infty) \quad \forall\, x \in O.
$$
\end{lemma}
The proof of the above lemma follows similar ideas to those in the proof of \cite[Theorem 1.1]{FGKU}, but with some modifications; see also the proof of Lemma \ref{lem_ent_new} above. 
\begin{proof}
We only prove the first claim in the lemma as the second claim follows by symmetry. Letting $m = 1, 2, \dots$ and applying the operator $\Delta_{g_1}^m$ to equation \eqref{u_1} with $j = 1$ and using the functional calculus, we get
\begin{equation}
\label{eq_300_1}
(-\Delta_{g_1})^{\alpha} \Delta_{g_1}^m u_1=\Delta_{g_1}^m(f-V_1u_1) \quad \text{on}\quad M_1.
\end{equation}
Similarly, applying $\Delta_{g_2}^m$ to the equation \eqref{u_2} with $j=2$, we obtain that   	
\begin{equation}
\label{eq_300_2}
(-\Delta_{g_2})^{\alpha} \Delta_{g_2}^m \tilde u_2=\Delta_{g_2}^m(f-V_2\tilde u_2) \quad \text{on}\quad M_2.
\end{equation}	
Since $\Delta_{g_1}^m(f-V_1u_1)\in C^\infty(M_1)$ and $\Delta_{g_2}^m(f-V_2\tilde u_2) \in C^\infty(M_2)$ are such that    
\[
(\Delta_{g_1}^m(f-V_1u_1), 1)_{L^2(M_1)}=0 \quad\text{and} \quad (\Delta_{g_2}^m(f-V_2\tilde u_2), 1)_{L^2(M_2)}=0,
\]
 we conclude from  \eqref{eq_300_1} and \eqref{eq_300_2} that 	
\begin{equation}
\label{eq_300_3}
\Delta_{g_1}^m u_1=(-\Delta_{g_1})^{-\alpha} (\Delta_{g_1}^m(f-V_1u_1)) \quad \text{on}\quad M_1,
\end{equation}	
and 	
\begin{equation}
\label{eq_300_4}
\Delta_{g_2}^m \tilde u_2=(-\Delta_{g_2})^{-\alpha} (\Delta_{g_2}^m(f-V_2\tilde u_2)) \quad \text{on}\quad M_2,
\end{equation}	
for $m = 1, 2, \dots$.  Using \eqref{u_12} and the fact that $g_1|_{O} = g_2|_{O}$, we obtain from \eqref{eq_300_3} and \eqref{eq_300_4} that
\begin{equation}
\label{eq_300_5}
\big((-\Delta_{g_1})^{-\alpha} \Delta_{g_1}^m(f-V_1u_1)\big)|_{O}= \big((-\Delta_{g_2})^{-\alpha} \Delta_{g_2}^m(f-V_2\tilde u_2)\big)|_{O}. 
\end{equation}
Using \eqref{frac_laplace_heat_inverse}, we get from \eqref{eq_300_5} that 
\begin{equation}
\label{eq_300_6}
	\int_0^\infty \bigg(\big(e^{t\Delta_{g_1}} \Delta_{g_1}^m(f-V_1u_1)\big)(x) - \big(e^{t\Delta_{g_2}} \Delta_{g_2}^m(f-V_2\tilde u_2)\big)(x)\bigg) \frac{1}{t^{1-\alpha}}dt=0,
\end{equation}
for $x\in O$, and $m=1,2,\dots$.  Note that the functions $e^{t\Delta_{g_1}} \Delta_{g_1}^m(f - V_1u_1) \in C^\infty([0,\infty); C^\infty(M_1))$ and $e^{t\Delta_{g_2}} \Delta_{g_2}^m(f - V_2\tilde{u}_2) \in C^\infty([0,\infty); C^\infty(M_2))$. In view of the bound \eqref{eq_100_-1}, the integral in \eqref{eq_300_6} converges uniformly for $x \in O$.

Using \eqref{eq_100_4}, we derive from  \eqref{eq_300_6}	
that 
\begin{equation}
\label{eq_300_7}
	\int_0^\infty \p_t^m\bigg(\big(e^{t\Delta_{g_1}} (f-V_1u_1)\big)(x) - \big(e^{t\Delta_{g_2}} (f-V_2\tilde u_2)\big)(x)\bigg) \frac{1}{t^{1-\alpha}}dt=0,
\end{equation}
for $x\in O$, and $m=1,2,\dots$.

Let $K := \supp(f) \subset O$ and let $\omega \subset\subset O$ be such that $K \cap \overline{\omega} = \emptyset$. Assuming that $x \in \omega$, we shall integrate by parts in \eqref{eq_300_7} $m$ times. We claim that there will be no contributions from the endpoints. Here, our assumption that $V_j|_O = 0$ is crucial.

To show that no contribution arises at $t = 0$ when integrating by parts in \eqref{eq_300_7}, using \eqref{eq_100_4}, we first write, for $t > 0$ and $x \in \omega$,
\begin{equation}
\label{eq_300_8}
\begin{aligned}
\p_t^l\big(e^{t\Delta_{g_1}} (f-V_1u_1)\big)(x)=&\int_{K}   e^{t\Delta_{g_1}}(x,y) (\Delta_{g_1}^l f)(y)dV_{g_1}(y)\\
&-
\int_{M_1\setminus O} e^{t\Delta_{g_1}}(x,y) (\Delta_{g_1}^l(V_1u_1)) (y)dV_{g_1}(y),
\end{aligned}
\end{equation}
for $l=0,1,\dots, m-1$.  Using \eqref{Gauss_heat}, we obtain from \eqref{eq_300_8} for $0<t<1$ and $x\in \omega$ that 
\begin{equation}
\label{eq_300_9}
\begin{aligned}
|\p_t^l\big(e^{t\Delta_{g_1}} (f-V_1u_1)\big)(x)|& \le  \|e^{t\Delta_{g_1}}(\cdot,\cdot)\|_{L^\infty(\omega\times K)}\|\Delta_{g_1}^l f\|_{L^1(M_1)}\\
&+  \|e^{t\Delta_{g_1}}(\cdot,\cdot)\|_{L^\infty(\omega\times (M\setminus O))}\|\Delta_{g_1}^l (V_1u_1)\|_{L^1(M_1)}\\
&\le Ce^{-\frac{\tilde c}{t}}(\|\Delta_{g_1}^l f\|_{L^1(M_1)}+ \|\Delta_{g_1}^l (V_1u_1)\|_{L^1(M_1)}),
\end{aligned}
\end{equation}
for $l=0,1,\dots, m-1$.  Here, $\tilde{c} > 0$ depends on $\textrm{dist}_{g}(K, \overline{\omega}) > 0$ and $\textrm{dist}_{g}(\overline{\omega}, M \setminus O) > 0$. The bound \eqref{eq_300_9} and a similar bound for $|\partial_t^l \big(e^{t\Delta_{g_2}} (f - V_2 \tilde{u}_2)\big)(x)|$ show that no contribution arises at $t = 0$ when integrating by parts in \eqref{eq_300_7}.

To demonstrate that no contribution arises at $t = +\infty$ when integrating by parts in \eqref{eq_300_7}, we use the stochastic completeness of the heat kernel,
\[
\int_{M} e^{t\Delta_{g_1}}(x,y) \, dV_g(y) = 1,
\]
for $t > 0$ and $x \in M$, see \cite[Ch. VIII, Theorem 5, page 191]{Chavel_book}. We conclude from \eqref{eq_300_8} that for $t > 0$ and $x \in \omega$
\begin{equation}
\label{eq_300_10}
|\p_t^l\big(e^{t\Delta_{g_1}} (f-V_1u_1)\big)(x)| \le \|\Delta_{g_1}^l f\|_{L^\infty(M_1)}+ \|\Delta_{g_1}^l (V_1u_1)\|_{L^\infty(M_1)},
\end{equation}
for $l = 0, 1, \dots, m - 1$. The bound \eqref{eq_300_10}, together with a similar bound for $|\partial_t^l \big(e^{t\Delta_{g_2}} (f - V_2 \tilde{u}_2)\big)(x)|$, shows that no contribution arises at $t = +\infty$ when integrating by parts in \eqref{eq_300_7} thanks to the presence of the factor $t^{-(1-\alpha)}$. 

Now integrating by parts $m$ times in \eqref{eq_300_7}, we obtain
\begin{equation}
\label{eq_300_11}
	\int_0^\infty \bigg(\big(e^{t\Delta_{g_1}} (f-V_1u_1)\big)(x) - \big(e^{t\Delta_{g_2}} (f-V_2\tilde u_2)\big)(x)\bigg) \frac{1}{t^{1+m-\alpha}}dt=0,
\end{equation}
for $x \in \omega$, and $m = 1, 2, \dots$. Rewriting \eqref{eq_300_11} as
\begin{equation}
\label{eq_300_12}
	\int_0^\infty \bigg(\big(e^{t\Delta_{g_1}} (f-V_1u_1)\big)(x) - \big(e^{t\Delta_{g_2}} (f-V_2\tilde u_2)\big)(x)\bigg) \frac{1}{t^{2+m-\alpha}}dt=0,
\end{equation}
for $x \in \omega$, and $m = 0, 1, 2, \dots$, and introducing the change of variables $s = 1/t$, we deduce that
\begin{equation}
\label{eq_300_13}
\int_0^\infty Q(s)s^{m}ds=0, \quad m=0,1,2,\dots, 
\end{equation}
where
\[
Q(s)=\big(\big(e^{\frac{1}{s}\Delta_{g_1}} (f-V_1u_1)\big)(x) - \big(e^{\frac{1}{s}\Delta_{g_2}} (f-V_2\tilde u_2)\big)(x)\big) s^{-\alpha},
\]
for $x \in \omega$.

Recalling estimates \eqref{eq_300_9} and \eqref{eq_300_10} with $l = 0$, we observe that for $s > 0$,
\begin{equation}
\label{eq_300_14}
|Q(s)| \le \mathcal{O}(1) \frac{e^{-cs}}{s^\alpha},
\end{equation}
where $c > 0$. Let
\[
\mathcal{F}(1_{[0,\infty)}Q)(\xi) = \int_0^\infty Q(s) e^{-i\xi s} \, ds
\]
be the Fourier transform of $1_{[0,\infty)}Q$. Thanks to the bound \eqref{eq_300_14}, an application of \cite[Theorem 3.3.7]{Lerner} demonstrates that the function $\mathcal{F}(1_{[0,\infty)}Q)(\xi)$ is holomorphic for $\text{Im}\ \xi < c$. 

It follows from \eqref{eq_300_13} that $\mathcal{F}(1_{[0,\infty)}Q)$ vanishes at $0$ with all derivatives, and hence, $Q(s) = 0$ for $s > 0$. Thus, letting 
\[
\rho(t, x) := \big(e^{t\Delta_{g_1}} (f - V_1 u_1)\big)(x) - \big(e^{t\Delta_{g_2}} (f - V_2 \tilde{u}_2)\big)(x),
\]
we get
\begin{equation}
\label{eq_300_15}
\rho(t, x) = 0, \quad t > 0, \quad x \in \omega.
\end{equation}
Note that the function $\rho|_{(0,\infty)\times \overline{O}} \in C^\infty((0,\infty) \times \overline{O})$, and since $g_1|_{O} = g_2|_{O}$, we see that $\rho$ satisfies the heat equation
\begin{equation}
\label{eq_300_16}
(\p_t-\Delta_{g_1})\rho=0\quad \text{in}\quad (0,\infty)\times O. 
\end{equation}
In view of \eqref{eq_300_16}, \eqref{eq_300_15}, and the fact that $O$ is connected, it follows by the unique continuation for the heat equation, see \cite[Sections 1 and 4]{Lin}, that
\begin{equation}
\label{eq_300_17}
\big(e^{t\Delta_{g_1}} (f-V_1u_1)\big)(x) =\big(e^{t\Delta_{g_2}} (f-V_2\tilde u_2)\big)(x),\quad t>0, \quad x\in O.
\end{equation}
Using \eqref{u_1} with $j = 1$ and \eqref{u_2} with $j = 2$, we obtain from \eqref{eq_300_17} that
\[
\big(e^{t\Delta_{g_1}} (-\Delta_{g_1})^\alpha u_1\big)(x) =\big(e^{t\Delta_{g_2}} (-\Delta_{g_2})^\alpha \tilde u_2\big)(x),\quad t>0, \quad x\in O,
\]
showing the claim of the lemma.
\end{proof}

We are ready to state the proof of Proposition~\ref{prop1}. In the proof of the proposition, we will use the notation $\{(\lambda_k^{(j)},\phi_{k,\ell}^{(j)})\}_{k=0}^{\infty}$, $d_k^{(j)}$, $\pi_k^{(j)}$ for $j=1,2$, to stand for the spectral decomposition of $-\Delta_{g_j}$ on $M_j$ as described in Section~\ref{sec_pre}. 
\begin{proof}[Proof of Proposition~\ref{prop1}]
	Note that $\lambda_0^{(1)}=\lambda_0^{(2)}=0$ and that both corresponding eigenfunctions are nontrivial constant functions on $M_j$, $j=1,2$. Therefore, the crux of the proof lies in proving similar claims for the strictly positive eigenvalues of $(M_j,g_j)$, $j=1,2$ and their corresponding eigenfunctions. To this end, let $f\in \mathcal H_{M_1,g_1,V_1}^O= \mathcal H_{M_2,g_2,V_2}^O$ and $u_1=S_{M_1,g_1,V_1}(f)$ and let $u_2=S_{M_2,g_2,V_2}(f).$  Also, let $\tilde{u}_1$, $\tilde{u}_2$ be as in \eqref{u_2} such that \eqref{u_12} is satisfied. We note that in view of Lemma~\ref{lem_reduce_to_heat} there holds
	\begin{equation}
		\label{u_12_heat_kernel}
	 (e^{t\Delta_{g_1}}(-\Delta_{g_1})^\alpha u_1)(x)=  (e^{t\Delta_{g_2}}(-\Delta_{g_2})^\alpha \tilde{u}_2)(x) \quad \forall\, t\in [0,\infty) \quad \forall\,x\in O,
	\end{equation}
and analogously that
	\begin{equation}
	\label{u_tilde_12_heat_kernel}
	(e^{t\Delta_{g_1}}(-\Delta_{g_1})^\alpha \tilde{u}_1)(x)=  (e^{t\Delta_{g_2}}(-\Delta_{g_2})^\alpha u_2)(x) \quad \forall\, t\in [0,\infty) \quad \forall\,x\in O.
\end{equation}
Writing the above two equalities in terms of the spectral representation of the heat kernel $e^{t\Delta_{g_j}}$ on $(M_j,g_j)$, we obtain
	\begin{equation}
	\label{u_12_spectral_eq}
\sum_{k=1}^{\infty}  (\lambda_k^{(1)})^\alpha e^{-\lambda_k^{(1)}t}\, (\pi_k^{(1)}\,u_1)(x)= \sum_{k=1}^{\infty}  (\lambda_k^{(2)})^\alpha e^{-\lambda_k^{(2)}t}  (\pi_k^{(2)}\,\tilde{u}_2)(x) \quad t\ge 0, \quad x\in O,
\end{equation}
and
	\begin{equation}
	\label{u_tilde_12_spectral_eq}
	\sum_{k=1}^{\infty}  (\lambda_k^{(1)})^\alpha e^{-\lambda_k^{(1)}t}\, (\pi_k^{(1)}\,\tilde{u}_1)(x)= \sum_{k=1}^{\infty}  (\lambda_k^{(2)})^\alpha e^{-\lambda_k^{(2)}t}  (\pi_k^{(2)}\,u_2)(x) \quad t\ge  0, \quad x\in O,
\end{equation}
We claim that each of the four series above is uniformly convergent for $x$ in the corresponding manifold $M_1$ or $M_2$ and $t \ge 0$ due to the smoothness of the functions $u_j$ and $\tilde{u}_j$ on $M_j$, $j = 1, 2$. Indeed, let us show this for the first series in \eqref{u_12_spectral_eq}; the proof for the others is similar. See also \cite[Section 2]{Fei21}.
 In doing so, 
we first recall that 
\begin{equation}
\label{eq_400_1}
\pi_k^{(1)}\,u_1=\sum_{l=1}^{d_k^{(1)}}(u_1,\phi_{k,l}^{(1)})_{L^2(M_1)}\,\phi_{k,l}^{(1)}.
\end{equation}
Using that $-\Delta_{g_1} \phi_{k,l}^{(1)} = \lambda_k^{(1)} \phi_{k,l}^{(1)}$ on $M_1$, for  $l=1,\dots, d_k^{(1)}$, $k=1,2,\dots$,  and the Cauchy--Schwarz inequality, we get
\begin{equation}
\label{eq_400_2}
|(u_1,  \phi_{k,l}^{(1)})_{L^2(M_1)}|=(\lambda_k^{(1)})^{-m}|((-\Delta_{g_1})^m u_1, \phi_{k,l}^{(1)})_{L^2(M_1)}|\le (\lambda_k^{(1)})^{-m}\|(-\Delta_{g_1})^m u_1\|_{L^2(M_1)},
\end{equation}
for any $m\in \N$.  It follows from \eqref{eq_400_1} and \eqref{eq_400_2} that 
\begin{equation}
\label{eq_400_1_new_1}
\|\pi_k^{(1)}\,u_1\|_{L^\infty(M_1)}\le (\lambda_k^{(1)})^{-m}\|(-\Delta_{g_1})^m u_1\|_{L^2(M_1)} \sum_{l=1}^{d_k^{(1)}}\|\phi_{k,l}^{(1)}\|_{L^\infty(M_1)}.
\end{equation}
To proceed, let us recall the following sup-norm estimate for $L^2$-normalized eigenfunctions: there exists a constant $C>0$ such that
\begin{equation}
\label{eq_400_2_2}
\| \phi_{l,k}^{(1)}\|_{L^\infty(M_1)} \le C(\lambda_k^{(1)})^{\frac{n-1}{4}},
\end{equation}
for all $\lambda_k^{(1)} \ge 1$; see \cite[Sections 3.2, formula (3.2.2)]{Sogge_book_eigenfunctions}. 
Furthermore, the following consequence of Weyl's law holds: there exists a constant $C > 0$ such that 
\begin{equation}
\label{eq_400_3_new_1}
N(\lambda) \le C \lambda^{\frac{n}{2}},
\end{equation}
for $\lambda \ge 1$ sufficiently large, where $N(\lambda)$ is the number of eigenvalues of $-\Delta_{g_1}$, counted with multiplicity, that are $\le \lambda$;  see \cite[Theorem 3.3.1]{Sogge_book_eigenfunctions}. 
It follows from \eqref{eq_400_3_new_1} that 
\begin{equation}
\label{eq_400_3}
d_k^{(1)} \le C (\lambda_{k}^{(1)})^{\frac{n}{2}}, \quad \lambda_k^{(1)} \ge C^{-\frac{2}{n}} k^{\frac{2}{n}},
\end{equation}
for $k\ge 1$ sufficiently large. 

Fixing $m \in \N$ such that $m - \alpha - \frac{(3n-1)}{4} \ge n$, and using \eqref{eq_400_1_new_1}, \eqref{eq_400_2_2}, and \eqref{eq_400_3}, we obtain that 
\begin{equation}
\label{eq_400_4}
\begin{aligned}
(\lambda_k^{(1)})^\alpha e^{-\lambda_k^{(1)} t}\, \| \pi_{k}^{(1)}u_1\|_{L^\infty(M_1)}  \le C (\lambda_k^{(1)})^\alpha (\lambda_k^{(1)})^{-m} d_k^{(1)}(\lambda_k^{(1)})^{\frac{n-1}{4}}\|(-\Delta_{g_1})^m u_1\|_{L^2(M_1)}\\
\le  C \frac{1}{k^{\frac{2}{n}(m - \alpha - \frac{(3n-1)}{4})}} \|(-\Delta_{g_1})^m u_1\|_{L^2(M_1)} \le \frac{C}{k^2}\|(-\Delta_{g_1})^m u_1\|_{L^2(M_1)}, \quad t \ge 0,
\end{aligned}
\end{equation}
for $k\ge 1$ sufficiently large.
The bound \eqref{eq_400_4} shows that the first series in \eqref{u_12_spectral_eq} converges uniformly for $x \in M_1$ and $t \ge 0$, establishing the claim.

By taking the Laplace transform 
$$\mathcal L(h)(z)= \int_0^{\infty} h(t)\, e^{-zt}\,dt, \quad \textrm{Re}(z)>0,$$
of the equations  \eqref{u_12_spectral_eq} and \eqref{u_tilde_12_spectral_eq} for each fixed $x\in O$, we arrive at the equalities
	\begin{equation}
	\label{resolvent_1}
	\sum_{k=1}^{\infty}  \frac{(\lambda_k^{(1)})^\alpha\,(\pi_k^{(1)}\,u_1)(x)}{\lambda_k^{(1)}+z}= \sum_{k=1}^{\infty}  \frac{(\lambda_k^{(2)})^\alpha\,(\pi_k^{(2)}\,\tilde{u}_2)(x)}{\lambda_k^{(2)}+z} \quad x\in O, \quad \textrm{Re}\,z>0.
\end{equation}
and
\begin{equation}
	\label{resolvent_2}
	\sum_{k=1}^{\infty}  \frac{(\lambda_k^{(1)})^\alpha\,(\pi_k^{(1)}\,\tilde{u}_1)(x)}{\lambda_k^{(1)}+z}= \sum_{k=1}^{\infty}  \frac{(\lambda_k^{(2)})^\alpha\,(\pi_k^{(2)}\,u_2)(x)}{\lambda_k^{(2)}+z} \quad x\in O, \quad \textrm{Re}\,z>0.
\end{equation}
Letting  $\Omega_j:=\C\setminus \bigcup_{k=1}^\infty\{-\lambda_k^{(j)}\}$, we define the functions, 
\[
 \mathcal R^{(j)}(z;x)=\sum_{k=1}^{\infty}  \frac{(\lambda_k^{(j)})^\alpha\,(\pi_k^{(j)}\,u_j)(x)}{\lambda_k^{(j)}+z}, \quad  z\in \Omega_j, \quad x\in M_j, 
 \]
and 
\[
 \widetilde{\mathcal R}^{(j)}(z;x)= \sum_{k=1}^{\infty}  \frac{(\lambda_k^{(j)})^\alpha\,(\pi_k^{(j)}\,\tilde{u}_j)(x)}{\lambda_k^{(j)}+z},\quad  z\in \Omega_j, \quad x\in M_j,
 \]
 for  $j=1,2$.  We claim that the functions $z \mapsto \mathcal{R}^{(j)}(z;x)$ and $z \mapsto \widetilde{\mathcal{R}}^{(j)}(z;x)$ are holomorphic in $\Omega_j$, with possibly simple poles at the points $z = -\lambda_k^{(j)}$, when $x \in M_j$, $j = 1, 2$. We shall show that the function $z \mapsto \mathcal{R}^{(1)}(z;x)$ is holomorphic in $\Omega_1$; the proof for the other functions is similar, see also \cite[Lemma 2.8]{Fei21}. In doing so, we first note that each term of the series in the definition of $\mathcal{R}^{(1)}(z;x)$ is holomorphic in $\Omega_1$, and therefore, the result will follow if we show that the series in the definition of $\mathcal{R}^{(1)}(z;x)$ converges uniformly on any compact subset of $\Omega_1$. To show this, let $K \subset \Omega_1$ be compact. We first claim that there exists a constant $c > 0$ such that
\begin{equation}
\label{eq_700_1}
|\lambda^{(1)}_k + z| \ge c \quad \text{for all } z \in K \text{ and all } k \in \mathbb{N}.
\end{equation}
Indeed, first note that for each fixed $k$, we have $\min_{z \in K} |\lambda^{(1)}_k + z| > 0$. Let $B(0, R)$ be an open ball centered at 0 with radius $R > 0$ such that $K \subset B(0, R)$. Since the set of eigenvalues $\{\lambda_k^{(1)}\}_{k=1}^\infty$ is discrete, only finitely many $-\lambda_k^{(1)}$ are in $\overline{B(0, R)}$. Therefore,
\[
\min_{\{-\lambda_k^{(1)} : -\lambda_k^{(1)} \in \overline{B(0, R)}\}} \min_{z \in K} |\lambda^{(1)}_k + z| > 0.
\]
For all $-\lambda_k^{(1)} \in \mathbb{C} \setminus \overline{B(0, R)}$, we have
\[
\min_{z \in K} |\lambda^{(1)}_k + z| = \text{dist}(K, -\lambda_k^{(1)}) \ge \text{dist}(K, \partial B(0, R)) > 0.
\]
Thus, the claim \eqref{eq_700_1} follows.

Now, using \eqref{eq_700_1}, similarly to \eqref{eq_400_4}, we obtain that for all $z\in K$,
\begin{equation}
\label{eq_700_2}
\frac{(\lambda_k^{(1)})^\alpha  \| \pi_{k}^{(1)}u_1\|_{L^\infty(M_1)} }{|\lambda_k^{(j)}+z|} \le \frac{C}{k^2}\|(-\Delta_{g_1})^m u_1\|_{L^2(M_1)},
\end{equation}
for $k \ge 1$ sufficiently large. Here, $m \in \mathbb{N}$ is fixed such that $m - \alpha - \frac{(3n-1)}{4} \ge n$. The bound \eqref{eq_700_2} shows that the series in the definition of $\mathcal{R}^{(1)}(z;x)$ converges uniformly on $K$, and therefore, the function $z \mapsto \mathcal{R}^{(1)}(z;x)$ is holomorphic in $\Omega_1$, proving the claim.

Now in view of the equalities \eqref{resolvent_1}-\eqref{resolvent_2} together with analytic continuation we have that for each $x\in O$ there holds
\begin{equation}\label{Resolvent_1_eq}
 \mathcal R^{(1)}(z;x) = \widetilde{\mathcal R}^{(2)}(z;x)  \quad z\in \C \setminus \bigcup_{k=1}^{\infty}\{-\lambda_k^{(1)},-\lambda_k^{(2)}\},
\end{equation}
and
\begin{equation}\label{Resolvent_2_eq}
	\widetilde{\mathcal R}^{(1)}(z;x) = {\mathcal R}^{(2)}(z;x) \quad z\in \C \setminus \bigcup_{k=1}^{\infty}\{-\lambda_k^{(1)},-\lambda_k^{(2)}\}.
\end{equation}

Next, using \eqref{Resolvent_1_eq} and \eqref{Resolvent_2_eq}, we shall show that
\begin{equation}
\label{eq_400_5}
\lambda_k^{(1)}=\lambda_k^{(2)}, \quad \text{and} \quad d_k^{(1)}=d_k^{(2)},
\end{equation}
for all $k=1,2,\dots$. To do this, let us start with $k=1$ and first assume that $\lambda_1^{(1)} \le \lambda_1^{(2)}$. Then, using \eqref{Resolvent_1_eq} and \eqref{resolvent_1}, we obtain for $x \in O$ that
\begin{equation}
\label{eq_400_6}
\begin{aligned}
(\lambda_1^{(1)})^{\alpha} (\pi_1^{(1)}u_1)(x) &= \lim_{z \to -\lambda_1^{(1)}} (z + \lambda_1^{(1)}) \mathcal{R}^{(1)}(z; x) = \lim_{z \to -\lambda_1^{(1)}} (z + \lambda_1^{(1)}) \widetilde{\mathcal{R}}^{(2)}(z; x) \\
&= 
\begin{cases} 
0, & \text{if} \quad \lambda_1^{(1)} \ne \lambda_1^{(2)}, \\
(\lambda_1^{(1)})^{\alpha} (\pi_1^{(2)}\tilde{u}_2)(x), & \text{if} \quad \lambda_1^{(1)} = \lambda_1^{(2)}.
\end{cases}
\end{aligned}
\end{equation}
Here we used that $\lambda_1^{(1)} \le \lambda_1^{(2)} < \lambda_2^{(2)} < \lambda_3^{(2)} < \dots$. Recalling that $u_1 = S_{M_1,g_1,V_1}(f)$ and using Lemma~\ref{density_lemma}, we see that there is $f \in \mathcal{H}_{M_1,g_1,V_1}^O$ such that $(u_1, \phi_{1,1}^{(1)})_{L^2(M_1)} \ne 0$. In view of 
\[
\pi_1^{(1)} u_1 = \sum_{l=1}^{d_1^{(1)}} (u_1, \phi^{(1)}_{1,l})_{L^2(M_1)} \phi_{1,l}^{(1)},
\]
and the fact that $\phi_{1,1}^{(1)}, \dots, \phi_{1,d_1^{(1)}}^{(1)}$ are linearly independent on $O$, 
we have that $\pi_1^{(1)} u_1 \ne 0$ on $O$. Therefore, \eqref{eq_400_6} implies that $\lambda_1^{(1)} = \lambda_1^{(2)}$. In the case when $\lambda_1^{(2)} \le \lambda_1^{(1)}$, we proceed similarly as above, but now using \eqref{Resolvent_2_eq} and \eqref{resolvent_2} instead of \eqref{Resolvent_1_eq} and \eqref{resolvent_1}. We obtain for $x \in O$ that
\begin{equation}
\label{eq_400_6_for_2}
\begin{aligned}
(\lambda_1^{(2)})^{\alpha} (\pi_1^{(2)} u_2)(x) &= \lim_{z \to -\lambda_1^{(2)}} (z + \lambda_1^{(2)}) \mathcal{R}^{(2)}(z; x) = \lim_{z \to -\lambda_1^{(2)}} (z + \lambda_1^{(2)}) \widetilde{\mathcal{R}}^{(1)}(z; x) \\
&=
\begin{cases} 
0, & \text{if} \quad \lambda_1^{(2)} \ne \lambda_1^{(1)}, \\
(\lambda_1^{(2)})^{\alpha} (\pi_1^{(1)} \tilde{u}_1)(x), & \text{if} \quad \lambda_1^{(2)} = \lambda_1^{(1)}.
\end{cases}
\end{aligned}
\end{equation}
As in the previous case, by choosing $f \in \mathcal{H}_{M_2,g_2,V_2}^O$ such that $(u_2, \phi_{1,1}^{(2)})_{L^2(M_2)} \ne 0$, we get $\pi_1^{(2)} u_2 \ne 0$ on $O$. Therefore, we conclude from \eqref{eq_400_6_for_2} that $\lambda_1^{(1)} = \lambda_1^{(2)}$ in this case as well. Proceeding similarly, by induction on $k = 2, 3, \dots$, we obtain that $\lambda_k^{(1)} = \lambda_k^{(2)}$.

With the equality $\lambda_k^{(1)} = \lambda_k^{(2)} =: \lambda_k$ established, we return to \eqref{eq_400_6} (as well as similar expressions with $\lambda_k$ and $\pi_k^{(j)}$, $k = 2, 3, \dots$, instead of $\lambda_1$ and $\pi_1^{(j)}$, $j = 1, 2$). We obtain that for $x \in O$,
\begin{equation}
\label{eq_400_7}
(\pi_k^{(1)} u_1)(x) = (\pi_k^{(2)} \tilde{u}_2)(x),
\end{equation}
where $u_1 = S_{M_1, g_1, V_1}(f)$ with $f \in \mathcal{H}_{M_1, g_1, V_1}^O$ being arbitrary. Here $k = 1, 2, \dots$. Note that for all $f \in \mathcal{H}_{M_1, g_1, V_1}^O$ such that $\pi_k^{(1)} u_1 \ne 0$, we have that $\pi_k^{(1)} u_1$ is an eigenfunction for $-\Delta_{g_1}$ on $M_1$ with the eigenvalue $\lambda_k$, and $\pi_k^{(2)} \tilde{u}_2$ is an eigenfunction for $-\Delta_{g_2}$ on $M_2$ with the eigenvalue $\lambda_k$. The equality \eqref{eq_400_7} shows that these eigenfunctions are equal on $O$. Letting  
\[
S := \textrm{Span}\{\pi_k^{(1)} u_1 : u_1 = S_{M_1, g_1, V_1}(f), f \in \mathcal{H}_{M_1, g_1, V_1}^O\} \subset \text{Ker}(-\Delta_{g_1} - \lambda_k),
\]
for all $k = 1, 2, \dots$, we next claim that
\begin{equation}
\label{eq_400_8}
S = \text{Ker}(-\Delta_{g_1} - \lambda_k).
\end{equation}
Indeed, to show the remaining inequality, we first note that if $\phi_{k, l}^{(1)} \in S$ for all $l = 1, \dots, d_k^{(1)}$, then the claim follows. If there exists $\phi_{k, l_0}^{(1)} \notin S$, then writing $\text{Ker}(-\Delta_{g_1} - \lambda_k) = S \oplus S^\perp$, we see that $\phi_{k, l_0}^{(1)} \in S^\perp$. Therefore,
\[
0 = (\pi_k^{(1)} u_1, \phi_{k, l_0}^{(1)})_{L^2(M_1)} = (u_1, \phi_{k, l_0}^{(1)})_{L^2(M_1)}
\]
for all $f \in \mathcal{H}_{M_1, g_1, V_1}^O$, which contradicts Lemma~\ref{density_lemma}, thereby proving the claim \eqref{eq_400_8}.

The equality \eqref{eq_400_8} shows that the set $\{\pi_k^{(1)} u_1 : u_1 = S_{M_1, g_1, V_1}(f), f \in \mathcal{H}_{M_1, g_1, V_1}^O\}$ contains $d_k^{(1)}$ linearly independent eigenfunctions for $-\Delta_{g_1}$ on $M_1$, associated with $\lambda_k$. According to \eqref{eq_400_7}, these eigenfunctions coincide on $O$ with $d_k^{(1)}$ eigenfunctions for $-\Delta_{g_2}$ on $M_2$, associated with $\lambda_k$. They should be linearly independent on $M_2$ because the linear dependence of eigenfunctions associated with a fixed eigenvalue is equivalent to the linear dependence of their restrictions on the set $O$, thanks to the unique continuation principle for elliptic equations. Hence, we get $d_k^{(2)} \geq d_k^{(1)}$.

Returning to \eqref{eq_400_6_for_2} (as well as similar expressions with $\lambda_k$ and $\pi_k^{(j)}$, $k = 2, 3, \dots$, instead of $\lambda_1$ and $\pi_1^{(j)}$, $j = 1, 2$), we obtain that for $x \in O$,
\begin{equation}
\label{eq_400_7_for_2}
(\pi_k^{(2)} u_2)(x) = (\pi_k^{(1)} \tilde{u}_1)(x),
\end{equation}
where $u_2 = S_{M_2, g_2, V_2}(f)$ with $f \in \mathcal{H}_{M_2, g_2, V_2}^O$ being arbitrary. Here $k = 1, 2, \dots$. Arguing similarly as after \eqref{eq_400_7}, now based on \eqref{eq_400_7_for_2}, we see that $d_k^{(1)} \geq d_k^{(2)}$, completing the establishment of \eqref{eq_400_5}.

Finally, to obtain the required $L^2(M_j)$ Schauder bases, we select $d_k^{(1)}$ linearly independent eigenfunctions for $-\Delta_{g_1}$ on $M_1$, associated with $\lambda_k$, from the set $\{\pi_k^{(1)} u_1 : u_1 = S_{M_1, g_1, V_1}(f), f \in \mathcal{H}_{M_1, g_1, V_1}^O\}$. The equality \eqref{eq_400_7} demonstrates the existence of $d_k^{(1)}$ linearly independent eigenfunctions for $-\Delta_{g_2}$ on $M_2$, associated with $\lambda_k$, of the form $\pi_k^{(2)} \tilde{u}_2$. This completes the proof. 
\end{proof}

\section{Variant of Gel'fand inverse spectral problem}
\label{sec_Gelfand}

The aim of this section is to give a proof of the variant of Gel'fand inverse spectral problem presented in Theorem~\ref{t2}. We will need the following lemma that are simpler instances of the more general observability estimates of \cite{BLR0,BLR} subject to the sharp so-called geometric control condition.
\begin{lemma}
	\label{lem_nontrapping}
	Let $(M,g)$ be a smooth closed and connected Riemannian manifold and assume that $O\subset M$ is a nonempty open set such that $M\setminus O$ is nonempty and nontrapping. There exists $C>0$ depending only on $(M,g)$ and $O$ such that
	$$ \|\phi\|_{L^2(M)} \leq C \|\phi\|_{L^2(O)},$$
	for any eigenfunction $\phi$ associated to $-\Delta_g$ on $M$.
	\end{lemma}
\begin{proof}

We start by proving that for all sufficiently large $T > 0$, the set $(0,T) \times O$ satisfies the geometric control condition of Bardos, Lebeau, and Rauch \cite{BLR0,BLR}, which states that all geodesic rays propagating in $M$ meet $O$ within time $T$. 

In doing so, we denote the unit tangent bundle of the manifold $M \setminus O$ by $S(M \setminus O)$, and for any $(x, v) \in S(M \setminus O)$, we denote by $\gamma_{x,v}$ the unique geodesic on $M \setminus O$ such that $\gamma_{x,v}(0) = x$ and $\dot{\gamma}_{x,v}(0) = v$. Here, we view $M \setminus O$ as a compact manifold with boundary embedded in the closed manifold $M$, and the geodesic vector field is defined on the whole of $M$. We consider the forward and backward exit time functions, defined by
\begin{align*}
&\tau_+: S(M \setminus O) \to [0, +\infty], \ \tau_+(x,v) = \sup\{t \ge 0: \gamma_{x,v}(s) \in M \setminus O\, \, \forall s \in [0, t]\},\\
&\tau_-: S(M \setminus O) \to [-\infty, 0], \ \tau_-(x,v) = -\sup\{t \ge 0: \gamma_{x,v}(s) \in M \setminus O\, \,  \forall s \in [-t, 0]\}.
\end{align*}
The fact that the manifold $M \setminus O$ is non-trapping means that $\tau_+(x,v) < +\infty$ and $-\tau_-(x,v) < +\infty$ for all $(x,v) \in S(M \setminus O)$. 

Now, if we show that
\begin{equation}
\label{T_def}
T_1 := \sup_{(x,v) \in S(M \setminus O)} (\tau_+(x,v) - \tau_-(x,v)) < +\infty,
\end{equation}
then the geometric control condition of Bardos, Lebeau, and Rauch \cite{BLR0,BLR} would be satisfied on the set $(0,T) \times O$ for all $T > T_1$. To prove \eqref{T_def}, suppose for the sake of contradiction that $T_1 = +\infty$. Then there exists a sequence $(x_k, v_k) \in S(M \setminus O)$ such that $\tau_+(x_k, v_k) - \tau_-(x_k, v_k) \to +\infty$ as $k \to \infty$. Since $S(M \setminus O)$ is compact, by passing to a subsequence, we may assume that $(x_k, v_k) \to (x_0, v_0) \in S(M \setminus O)$ as $k \to \infty$.  As the exit time functions $\tau_+$ and $-\tau_-$ are upper semi-continuous (see \cite[Section 2.1, page 538]{Guillarmou_Mazzucchelli_Tzou_2021}, see also \cite[Lemma 6.3]{Feizmohammadi_Lassas_Oksanen_2021}), we get
\[
\lim_{k \to \infty} (\tau_+(x_k, v_k) - \tau_-(x_k, v_k)) \leq \tau_+(x_0, v_0) - \tau_-(x_0, v_0)<+\infty.
\]
This contradiction shows the claim \eqref{T_def}.

Let us now fix some $T > T_1$. It is known that since $O$ is an open set, the observability for the wave equation holds when the set $(0,T) \times O$ satisfies the geometric control condition on $M$,  (see the main result of \cite{BLR0}, see also \cite[Theorem A.4]{BFMO} and \cite{Humbert_Privat_Trelat_2019}). Observability states that there exists $C > 0$ depending only on $(M,g)$, $O$, and $T$ such that
\begin{equation}
\label{control_bound}
\|u|_{t=0}\|_{L^2(M)} + \|\partial_t u|_{t=0}\|_{H^{-1}(M)} \leq C \|u\|_{L^2((0,T) \times O)}
\end{equation}
for any $u \in C^1((0,T);H^{-1}(M)) \cap C((0,T);L^2(M))$ that satisfies the wave equation
$$
\partial_t^2 u - \Delta_g u = 0 \quad \text{on } (0,T) \times M.
$$
Let $\phi \in C^{\infty}(M)$ be an eigenfunction for $-\Delta_g$ on $M$, namely
$$
-\Delta_g \phi = \lambda \phi \quad \text{on } M,
$$
for some $\lambda \geq 0$. Subsequently, define $u(t,x) = \cos(\sqrt{\lambda} t) \phi(x)$. Applying the estimate \eqref{control_bound} to the function $u$, we obtain
$$
\|\phi\|_{L^2(M)} \leq C \|\cos(\sqrt{\lambda} t)\|_{L^2((0,T))} \|\phi\|_{L^2(O)} \leq C \sqrt{T} \|\phi\|_{L^2(O)}.$$
The result follows.
\end{proof}

We also need the following lemma. Its proof is similar to that of Lemma \ref{lem_analog_5_2}, but instead of relying on the unique continuation principle for the fractional Laplacian, it relies on the unique continuation principle for elliptic operators. Therefore, the proof is omitted.
\begin{lemma}
	\label{lem_fixed_eigenspace}
	Let $(M,g)$ be a smooth closed and connected Riemannian manifold and assume that $\mathcal V\subset M$ is a nonempty open set. Let $\lambda\geq 0$ be an eigenvalue for $-\Delta_g$ on $M$ with multiplicity $d\geq 1$ and let $E_\lambda=\textrm{span}\,\{\theta_1,\ldots,\theta_d\}$ be the eigenspace associated to $\lambda$. Given any $c=(c_1,\ldots,c_d)\in \C^d$, there exists $f\in C^{\infty}_0(\mathcal V)$ such that $(f,\theta_\ell)_{L^2(\mathcal V)}=c_\ell$ for $\ell=1,\ldots,d$.
\end{lemma}

\begin{proof}[Proof of Theorem~\ref{t2}]
Let $p \in O$ be as in (H) and subsequently consider a fixed $q^{(1)} \in \mathcal A_{M_1,g_1}(p)$ and also a fixed $q^{(2)}\in \mathcal A_{M_2,g_2}(p)$. We may assume, without loss of generality, that
\begin{equation}
	\label{pq_rel}
	\textrm{dist}_{g_1}(p,q^{(1)}) \geq \textrm{dist}_{g_2}(p,q^{(2)}),
\end{equation}
as the other case can be dealt with analogously. 

For the sake of simplicity in the presentation, we will assume without any loss of generality that $\{\psi^{(1)}_k\}_{k=0}^{\infty}$ is an $L^2(M_1)$-orthonormal basis. By applying the Gram-Schmidt orthogonalization procedure to $\{\psi^{(1)}_k\}_{k=0}^{\infty}$ and simultaneously performing the same operations on $\{\psi^{(2)}_k\}_{k=0}^{\infty}$ at each step, we can always reduce to the case where $\{\psi^{(1)}_k\}_{k=0}^{\infty}$ forms an $L^2(M_1)$-orthonormal basis. Specifically, by performing the same operations, we mean the following: at the first step of the Gram-Schmidt orthogonalization procedure, we let
\begin{equation}
\label{eq_101_1}
\widetilde{\psi}^{(1)}_0 = \frac{\psi^{(1)}_0}{\|\psi^{(1)}_0\|_{L^2(M_1)}}, \quad \widetilde{\psi}^{(2)}_0 = \frac{\psi^{(2)}_0}{\|\psi^{(1)}_0\|_{L^2(M_1)}}.
\end{equation}
Now, proceeding by induction, at the $k$-th step, we set
\begin{equation}
\label{eq_101_2}
\widehat{\psi}^{(1)}_k = \psi^{(1)}_k - \sum_{l=0}^{k-1} c_{kl} \widetilde{\psi}_l^{(1)},
\end{equation}
where 
\[
c_{kl} = (\psi^{(1)}_k, \widetilde{\psi}^{(1)}_l)_{L^2(M_1)}, \quad l = 0, 1, \dots, k-1,
\]
for $k = 1, 2, \dots$. We observe that $(\widehat{\psi}^{(1)}_k, \widetilde{\psi}^{(1)}_l)_{L^2(M_1)} = 0$ for all $l = 0, 1, \dots, k-1$, and $\widehat{\psi}^{(1)}_k \neq 0$ as $\psi^{(1)}_k \notin \text{Span}\{\widetilde{\psi}_0^{(1)}, \dots, \widetilde{\psi}_{k-1}^{(1)}\} = \text{Span}\{\psi_0^{(1)}, \dots, \psi_{k-1}^{(1)}\}$, for $k = 1, 2, \dots$. We then let
\begin{equation}
\label{eq_101_3}
\widetilde{\psi}^{(1)}_k = \frac{\widehat{\psi}^{(1)}_k}{\|\widehat{\psi}^{(1)}_k\|_{L^2(M_1)}}, \quad \widetilde{\psi}^{(2)}_k = \frac{\psi^{(2)}_k - \sum_{l=0}^{k-1} c_{kl} \widetilde{\psi}_l^{(2)}}{\|\widehat{\psi}^{(1)}_k\|_{L^2(M_1)}},
\end{equation}
for $k = 1, 2, \dots$. It follows that $\{\widetilde{\psi}^{(1)}_k\}_{k=0}^{\infty}$ forms an $L^2(M_1)$-orthonormal basis. Furthermore, in view of \eqref{eq_101_1}, \eqref{eq_101_2}, and \eqref{eq_101_3}, the equality \eqref{psi_eq} implies that 
\[
\widetilde{\psi}^{(1)}_k(x) = \widetilde{\psi}^{(2)}_k(x), \quad x \in O,
\]
 for $k = 0, 1, 2, \dots$. We also note that at the end of this procedure, the (Schauder) basis $\{\widetilde{\psi}^{(2)}_k\}_{k=0}^{\infty}$ is not necessarily orthonormal. Despite this, our aim for the remainder of the proof is to show that the latter basis is actually also an orthonormal basis of $L^2(M_2)$. In what follows, we shall return to denoting the bases by $\{\psi^{(1)}_k\}_{k=0}^{\infty}$ and $\{\psi^{(2)}_k\}_{k=0}^{\infty}$.

We start by claiming that there exists $\varepsilon_0 \in (0,1)$ sufficiently small such that for all $0 < \varepsilon \le \varepsilon_0$, the following holds:
\begin{itemize}
\item[{\bf(P)}]{if $x \in \{y \in M_2 \,:\, \textrm{dist}_{g_2}(p,y) \geq \textrm{dist}_{g_2}(p,q^{(2)}) - \varepsilon\}$ then $x \in O$.}
\end{itemize}
To show (P), we suppose, for the sake of contradiction, that the claim is not true. Then, there exists a sequence $0<\varepsilon_k\le \frac{1}{k}$ and a sequence of points $x_k \in M_2$, $k \in \mathbb{N}$, with the property that
\[ 
\textrm{dist}_{g_2}(p, x_k) \geq \textrm{dist}_{g_2}(p, q^{(2)}) - \varepsilon_k \quad \text{and} \quad x_k \notin O \quad \forall k \in \mathbb{N}. 
\]
As $M_2 \setminus O$ is a closed set and $M_2$ is compact, it follows that $M_2 \setminus O$ is also compact. Therefore, there exists $x_0 \in M_2 \setminus O$ and a subsequence $x_{k_l}$ of the sequence $x_k$ such that $x_{k_l} \to x_0$ as $l \to \infty$. Hence,
\[ 
\textrm{dist}_{g_2}(p, x_0) = \lim_{l \to \infty} \textrm{dist}_{g_2}(p, x_{k_l}) \geq \textrm{dist}_{g_2}(p, q^{(2)}). 
\]
But then, by the definition of antipodal sets, it follows that $x_0 \in \mathcal{A}_{M_2, g_2}(p)$ and consequently via (H) that $x_0 \in O$, a contradiction. This shows the claim (P).

Given any $x\in O$ and any $r>0$ sufficiently small, let us denote by $\mathbb B_r(x)$ the open geodesic ball centred at the point $x\in O$ with radius $r$. Note that there is no ambiguity in the definition of small geodesic balls on the set $O$ as $g_1|_{O}=g_2|_{O}$. For the remainder of this proof, we will fix a $\varepsilon \in (0,1)$ sufficiently small so that 
\begin{equation}\label{balls_cond} \overline{\mathbb B_{\varepsilon}(q^{(1)}) \cup \mathbb B_{\varepsilon}(p)} \subset O,\end{equation}
and also that property (P) above is satisfied. Let us define open  connected sets $\mathcal V,\mathcal W \subset O$ by 
$$ \mathcal W= \mathbb B_{\frac{\varepsilon}{2}}(p) \quad \text{and} \quad  \mathcal V= \mathbb B_{\frac{\varepsilon}{2}}(q^{(1)}).$$
Let $f_1\in C^{\infty}_0(\mathcal V)$ be arbitrary. As $\{\psi^{(1)}_k\}_{k=0}^{\infty}$ is an $L^2(M_1)$-orthonormal basis, there holds
\begin{equation}
	\label{f_basis_rep}
	f_1(x) = \sum_{k=0}^{\infty} a_k \, \psi^{(1)}_k(x), \qquad x\in M_1,
\end{equation} 
where given any $k=0,1,2,\ldots,$
\begin{equation}\label{a_k_def} a_k = (f_1,\psi_k^{(1)})_{L^2(\mathcal V)}=\int_{\mathcal V} f_1(x)\, \overline{\psi^{(1)}_k}(x) \,dV_{g_1}=\int_{\mathcal V} f_1(x)\, \overline{\psi^{(2)}_k}(x) \,dV_{g_2},\end{equation}
and the convergence in \eqref{f_basis_rep} is to be understood in the $L^2(M_1)$--topology. Now, as $f \in H^s(M_1)$ for all $s \ge 0$, \eqref{f_basis_rep} holds with convergence in the $H^s(M_1)$--norm (see \cite[Chapter 7, Section 10]{Taylor_book_vol_II}). By Sobolev's embedding, the series in \eqref{f_basis_rep} converges uniformly on $M_1$ to $f_1$.

Using \eqref{eq_prelim_Sobolev_norms} and the fact that $\{\psi^{(1)}_k\}_{k=0}^{\infty}$ forms an $L^2(M_1)$-orthonormal basis, we obtain that the coefficients $a_k$ satisfy the bounds
\begin{equation}\label{a_k_bound}
\begin{aligned}
	|a_k| \leq \|f_1\|_{H^{r}(\mathcal V)}\, \|\psi_k^{(1)}\|_{H^{-r}(M_1)}\leq c_r\|f_1\|_{H^{r}(\mathcal V)}\, \|(I-\Delta_{g_1})^{-\frac{r}{2}}\psi_k^{(1)}\|_{L^2(M_1)}\\
	\leq c_r \,(1+\mu_k^{(1)})^{-\frac{r}{2}}\,\|f_1\|_{H^r(\mathcal V)},
\end{aligned}
\end{equation}
for any $r\geq 0$, where $c_r>0$ depends only on $(M_1,g_1)$ and $r$. For future reference, we also record that 
\begin{equation}
	\label{mu_1_bound}
	\sum_{k=0}^{\infty} \frac{1}{(1+\mu_k^{(1)})^{\frac{r}{2}}} <\infty \qquad \forall\, r>n.
\end{equation}
The claim \eqref{mu_1_bound} follows from the Weyl law bounds for the eigenvalues \eqref{eq_400_3}. 

Let us (for now formally) define the function
\begin{equation}
	\label{tilde_f}
	f_2(x) = \sum_{k=0}^{\infty} a_k \,\psi^{(2)}_k(x) \qquad x\in M_2,
\end{equation}
where $a_k$ are defined by  \eqref{a_k_def}.
We claim that the definition above is well-posed in the sense of the convergence of the right hand side with respect to the $H^{s}(M_2)$ norm for any $s\geq 0$ and that $f_2\in C^{\infty}(M_2)$. Indeed, let us first observe that owing to \eqref{psi_eq} there holds
\begin{equation}
\label{eq_500_1} 
\|\psi^{(2)}_k\|_{L^2(O)}=  \|\psi^{(1)}_k\|_{L^2(O)} \leq  \|\psi^{(1)}_k\|_{L^2(M_1)}=1 \quad k=0,1,\ldots.
\end{equation}
Applying Lemma~\ref{lem_nontrapping} with $M=M_2$ and $g=g_2$, the inequality \eqref{eq_500_1} reduces to 
\begin{equation}
\label{eq_500_2}
\|\psi^{(2)}_k\|_{L^2(M_2)} \leq C\, \quad k=0,1,\ldots,
\end{equation}
for some $C>0$ independent of $k$. Next, for any $m=1,2,\dots$, using that $(I-\Delta_{g_2})^m \psi^{(2)}_k=(1+\mu_k^{(2)})^m  \psi^{(2)}_k$ on $M_2$, the bounds \eqref{eq_prelim_Sobolev_norms}, and \eqref{eq_500_2}, we get 
\begin{equation}
\label{eq_500_3}
\|\psi^{(2)}_k\|_{H^{2m}(M_2)} \leq C (1+\mu_k^{(2)})^m\|\psi^{(2)}_k\|_{L^2(M_2)}\le C(1+\mu_k^{(2)})^m,
\end{equation}
for some $C=C_m>0$ independent of $k$. Letting $s\ge 0$ and interpolating between the bounds \eqref{eq_500_2} and \eqref{eq_500_3}, we obtain that 
\begin{equation}
	\label{psi_est_1}
	\|\psi^{(2)}_k\|_{H^s(M_2)} \leq C\, (1+\mu^{(2)}_k)^{\frac{s}{2}}, \quad k=0,1,\ldots, 
\end{equation}
where $C=C_s>0$ is independent of $k$, see \cite[Theorem 7.22]{Grubb_book}. For each fixed $s\geq 0$, recalling that $\mu_k^{(1)}=\mu_k^{(2)}$  and combining \eqref{psi_est_1} and \eqref{a_k_bound} with some fixed $r>n+s$,  we get
\begin{equation}
\label{eq_500_4}
  \|a_k\psi_k^{(2)}\|_{H^s(M_2)} \leq c_{r}\,C_s\, \|f_1\|_{H^r(\mathcal V)}\frac{1}{(1+\mu^{(1)}_k)^{\frac{r-s}{2}}}.
 \end{equation}
The bound \eqref{eq_500_4}, together with \eqref{mu_1_bound}, shows that the series in \eqref{tilde_f} converges on $M_2$ with respect to the $H^s(M_2)$-norm for any $s \geq 0$. Therefore, by Sobolev's embedding, the series in \eqref{tilde_f} converges uniformly on $M_2$ to $f_2$, and thus $f_2 \in C^\infty(M_2)$.

Let us also emphasize that the equality \eqref{psi_eq}, together with the uniform convergence of the series in \eqref{f_basis_rep} and \eqref{tilde_f}, implies that
\begin{equation}\label{tildefonO}
f_2(x) = f_1(x) \qquad \forall\, x \in O.
\end{equation}
Our goal for the remainder of this proof is to show that $f_2$ must vanish identically on $M_2 \setminus O$. To this end, let us define (for now formally)
\begin{equation}
    \label{U_j_def}
    U_j(t,x) = \sum_{k=0}^{\infty} a_k \cos(\sqrt{\mu_k^{(j)}} t) \psi_k^{(j)}(x), \qquad t \in \mathbb{R}, \quad x \in M_j, \quad j=1,2,
\end{equation}
where $\{a_k\}_{k=0}^{\infty}$ are as defined by \eqref{a_k_def}. We remark that the formal definition \eqref{U_j_def} is well-posed in the sense of convergence with respect to the $C^l(\mathbb{R}; H^s(M_j))$ norm for all $l \ge 0$ and $s \ge 0$. This follows analogously to the proof of the well-posedness of the definition \eqref{tilde_f} and therefore is omitted here. It is straightforward to see that for each $j=1,2,$ the function $U_j\in C^{\infty}(\R\times M_j)$ is the unique solution to the wave equation
\begin{equation}\label{wave_pf}
	\begin{aligned}
		\begin{cases}
			\p_t^2 U_j - \Delta_{g_j} U_j=0 
			&\text{on  $\R\times M_j$},
			\\
			U_j(0,x)=f_j(x) & \text{on $M_j$},
			\\
			\p_t U_j(0,x)=0 & \text{on $M_j$}.
		\end{cases}
	\end{aligned}
\end{equation}
In particular, we observe that the following equality is satisfied,
\begin{equation}
	\label{U_j_equality}
	U_1(t,x)=U_2(t,x), \qquad t\in \R, \quad x\in O,
 \end{equation}
thanks to the definition \eqref{U_j_def} together with the equality \eqref{psi_eq}. Next, we recall that $f_1\in C^{\infty}_0(\mathcal V)$. Therefore, using the fact that
$$ \textrm{dist}_{g_1}(\p \mathcal V,\p \mathcal W) \geq \textrm{dist}_{g_1}(p,q^{(1)})-\varepsilon.$$
together with finite speed of propagation for the wave equation \eqref{wave_pf} with $j=1$, it follows that
\begin{equation}
	\label{U_1_vanish}
	U_1(t,x)=0, \quad t \in [-T,T], \quad x\in \mathcal W,
\end{equation}
where 
$$ T:=  \textrm{dist}_{g_1}(\p \mathcal V,\p \mathcal W)\geq \textrm{dist}_{g_1}(p,q^{(1)})-\varepsilon\geq \textrm{dist}_{g_2}(p,q^{(2)})-\varepsilon,$$
where we have used \eqref{pq_rel} in the last inequality above. In view of \eqref{U_1_vanish} and \eqref{U_j_equality} we deduce that
\begin{equation}
	\label{U_2_vanish}
	U_2(t,x)=0, \quad t \in [-T,T], \quad x\in \mathcal W.
\end{equation}
Applying the unique continuation result of Tataru (see \cite{Tataru} and \cite{KKL}), it follows that
$$ f_2(x)=U_2(0,x)=0 \quad \text{for all $x\in M_2$ that satisfy $\textrm{dist}_{g_2}(x,p)\leq \textrm{dist}_{g_2}(p,q^{(2)})-\varepsilon$}.$$
Combining the previous equality with the property (P) we deduce that \begin{equation}\label{f_2_vanish}
	f_2(x)=0, \qquad x \in M_2\setminus O.
	\end{equation} 
It follows from \eqref{f_2_vanish} and \eqref{tildefonO} together with \eqref{a_k_def} and \eqref{psi_eq} that given any function $f\in C^{\infty}_0(\mathcal V)$ there holds:
\begin{equation}\label{eigen_expansion}
f(x)=\sum_{k=0}^{\infty} (f,\psi^{(2)}_k)_{L^2(\mathcal V)}\, \psi_k^{(2)}(x) \qquad x\in M_2,
\end{equation}
where we remind the reader that the convergence of the infinite series above holds in any $H^s(M_2)$-norm, $s\geq 0$.

We claim that equation \eqref{eigen_expansion} implies that $\{\psi_k^{(2)}\}_{k=0}^{\infty}$ must be an orthonormal basis of $L^2(M_2)$. To show this claim, we will rename the ordered Schauder basis $\{\psi_k^{(2)}\}_{k=0}^{\infty}$ in terms of the multiplicity of the eigenvalues, thus renaming it as the ordered set \begin{equation}\label{new_basis}\{\theta_{k,1},\ldots,\theta_{k,d_k}\}_{k=0}^{\infty}.\end{equation}
We emphasize here that we are merely renaming the basis and not changing its order. Equation \eqref{eigen_expansion} can now be restated as saying that given any $f\in C^{\infty}_0(\mathcal V)$, there holds
\begin{equation}\label{new_eigen_expansion}
f(x) = \sum_{k=0}^{\infty} \sum_{\ell=1}^{d_k} (f,\theta_{k,\ell})_{L^2(\mathcal V)}\, \theta_{k,\ell}(x), \quad x\in M_2.
\end{equation}
Our aim is then to show that \eqref{new_eigen_expansion} implies that for any fixed integer $k\geq 0$ the set $\{\theta_{k,1},\ldots,\theta_{k,d_k}\}$ is an $L^2(M_2)$-orthonormal set. We observe that, in view of \eqref{new_eigen_expansion} and the fact that eigenfunctions of $-\Delta_{g_2}$ associated with distinct eigenvalues are always $L^2(M_2)$-orthogonal, we have
\begin{equation}
	\label{final_eigen_expansion}
	(f,\theta_{k,m})_{L^2(\mathcal V)}= \sum_{\ell=1}^{d_k} (f,\theta_{k,\ell})_{L^2(\mathcal V)}\, (\theta_{k,\ell},\theta_{k,m})_{L^2(M_2)},
\end{equation}
for any $f\in C^{\infty}_0(\mathcal V)$, any $k\geq 0$ and any $m=1,\ldots,d_k$. Applying Lemma~\ref{lem_fixed_eigenspace} with $M=M_2$ and $g=g_2$, it follows that
$$ (\theta_{k,m},\theta_{k,\ell})_{L^2(M_2)}= \delta_{ml} \quad k=0,1,\ldots \quad m=0,\ldots,d_k,$$
where $\delta_{ml}$ is the Kronecker delta function. This shows that the set \eqref{new_basis} and thus $\{\psi^{(2)}_k\}_{k=0}^{\infty}$ is an $L^2(M_2)$-orthonormal basis.

We have arrived at the equality of two normalized spectral data subject to the observation set $O$ for the two manifolds $(M_j,g_j)$, $j=1,2$. The claim in Theorem~\ref{t2} now follows from the standard Gel'fand inverse spectral problem on closed manifolds, see \cite[Corollary 2]{HLOS_2018}, and see also \cite{KrKaLa} and \cite[Section 3]{Fei21}. 
\end{proof}

\section{Proof of Theorem~\ref{t1}}
\label{sec_proofs}

We have already proven in Section~\ref{sec_reduction_to_spectral_data} the reduction from  the Cauchy data set $\mathcal C_{M,g,V}^O$ to the knowledge of a non-normalized (Schauder) basis consisting of eigenfunctions associated to $-\Delta_g$ on $M$ restricted to the set $O$. In Section~\ref{sec_Gelfand} we showed that the latter data determines the isometry class of the manifold. Therefore, the only remaining unknown is the zeroth order coefficient $V$. The determination of lower order coefficients in various nonlocal elliptic equations are well-known. Indeed, as introduced in the work \cite{GSU}, the idea is to use the unique continuation principle $(\textrm{UCP})^\prime$ to obtain a Runge type approximation property result for solutions to the nonlocal equation. The recovery of the zeroth order coefficient $V$ then follows from an integration by parts technique that is very common in inverse problems that relates the Cauchy data set to certain knowledge of products of solutions to the nonlocal equation. We will also refer the reader to \cite[Section 5.3--Section 6]{Ghosh_Lin_Xiao_2017} for a presentation of the above mentioned idea in the context of variable leading order coefficients. We will present here a slightly different approach that is just based on the unique continuation principle $(\textrm{UCP})^\prime$.

\begin{proof}[Proof of Theorem~\ref{t1}]

In view of Proposition~\ref{prop1} and Theorem~\ref{t2}, there exists a smooth diffeomorphism $\Phi: M_1\to M_2$ such that $\Phi|_{O}$ is the identity map and additionally there holds $g_1= \Phi^\star g_2$ on $M_1$. Let us now define the function $\tilde V_2 \in C^{\infty}(M_1)$ by   
$$
\tilde V_2(x) = V_2(\Phi(x)), \quad \forall\, x\in M_1.
$$
Our goal is to prove that $V_1$ must be equal to $\tilde{V}_2$ on $M_1$. To this end, first, by Lemma~\ref{lem_obstruction}, we observe that
\begin{equation}
\label{eq_600_0}
\mathcal{C}_{M_2, g_2, V_2}^O = \mathcal{C}_{M_1, \Phi^\star g_2, \tilde{V}_2}^O = \mathcal{C}_{M_1, g_1, \tilde{V}_2}^O.
\end{equation}
Combining the equality \eqref{eq_600_0} with $\mathcal C_{M_2,g_2,V_2}^O=\mathcal C_{M_1,g_1,V_1}^O$, we deduce that
\begin{equation}\label{new_cauchy_eq} \mathcal C_{M_1,g_1,V_1}^O = \mathcal C_{M_1,g_1, \tilde{V}_2}^O.\end{equation}

Let $f \in \mathcal{H}_{M_1, g_1, V_1}^O$ be a non-zero function, which we know exists according to Remark~\ref{rem_structure_H_set} and Definition~\ref{def_S}. Define $u_1 = S_{M_1, g_1, V_1}(f) \in C^\infty(M_1)$. Then $u_1$ satisfies the equation
\begin{equation}
	\label{u_1_final_eq}
	(-\Delta_{g_1})^\alpha u_1 + V_1 u_1 = f \quad \text{on } M_1.
\end{equation}
The equality \eqref{new_cauchy_eq} implies that there exists some $u_2 \in C^\infty(M_1)$ that solves the equation
\begin{equation}
\label{eq_600_3}
(-\Delta_{g_1})^\alpha u_2 + \tilde{V}_2 u_2 = 0 \quad \text{on } M_1 \setminus \overline{O},
\end{equation}
with the additional property that
\begin{equation}
\label{eq_600_4}
	(u_1 - u_2)|_O = 0, \quad ((-\Delta_{g_1})^\alpha (u_1 - u_2))|_O = 0.
\end{equation}
By the unique continuation principle $(\textrm{UCP})'$ (recall that this follows from Theorem~\ref{thm_ent} with $N=1$), we conclude from \eqref{eq_600_4} that:
\begin{equation}
\label{eq_600_5}
 u_1 = u_2 \quad \text{on} \quad  M_1.
\end{equation}
Using \eqref{eq_600_5}, we obtain from \eqref{u_1_final_eq} and \eqref{eq_600_3} that
\[
(\tilde{V}_2(x) - V_1(x)) \, u_1(x) = 0 \quad \forall x \in M_1 \setminus \overline{O}.
\]
Recall that $\tilde{V}_2|_O = V_1|_O$. To establish that $\tilde{V}_2$ is identical to $V_1$ on $M_1$, it suffices to show that the set
\[
\mathbb{D} = \{ x \in M_1 \setminus \overline{O} \,:\, u_1(x) \neq 0 \}
\]
is dense in $M_1 \setminus \overline{O}$. We will prove this via contradiction. Suppose, to the contrary, that there exists an open nonempty set $\mathcal{U} \subset M_1 \setminus \overline{O}$ such that $\mathbb{D} \cap \mathcal{U} = \emptyset$, and therefore, $u_1|_\mathcal{U} = 0$. It follows from \eqref{u_1_final_eq} that $(( -\Delta_{g_1} )^\alpha u_1)|_\mathcal{U} = 0$. This implies that $u_1 = 0$ on $M_1$ thanks to $(\textrm{UCP})'$, and thus $f$ must be zero. This contradiction shows that $\mathbb{D}$ is dense in $M_1 \setminus \overline{O}$, completing the proof.
\end{proof}

\begin{remark}
Note that once the smooth diffeomorphism $\Phi: M_1 \to M_2$ such that $\Phi|_O = \text{Id}$ and $g_1 = \Phi^\star g_2$ on $M_1$ is recovered, the assumption that $V_1|_O = 0$ and $V_2|_O = 0$ is not needed to show that $V_1 = V_2 \circ \Phi$.
\end{remark}

\begin{proof}[Proof of Corollary~\ref{cor1}]
	The corollary follows immediately once we show that, under the hypotheses of Corollary~\ref{cor1}, the condition (H) is satisfied for $(M_j, g_j)$ with $j=1, 2$. Note that $(M_j \setminus O, g_j)$ is nontrapping by definition of a simple manifold.

For $j=1, 2$, we consider a small neighborhood $\mathcal{U}_j$ of $M_j \setminus O$ in $M_j$ so that $(\overline{\mathcal{U}_j}, g_j)$ is a simple Riemannian manifold and additionally that $\partial \mathcal{U}_1 = \partial \mathcal{U}_2$, see \cite[Proposition 3.8.7]{Paternain_Salo_Uhlmann_book}. To conclude the proof, we will show that for any point $p \in \partial \mathcal{U}_1 = \partial \mathcal{U}_2 \subset O$, we have $\mathcal{A}_{M_j, g_j}(p) \subset O$ for both $j=1, 2$. Assume, on the contrary, that there exists a point $p \in \partial \mathcal{U}_1 = \partial \mathcal{U}_2$ such that $\mathcal{A}_{M_j, g_j}(p) \cap (M_j \setminus O) \ne \emptyset$ for some $j \in \{1, 2\}$. Let $q^{(j)} \in \mathcal{A}_{M_j, g_j}(p) \cap (M_j \setminus O)$ and let us denote by $\gamma_j$ the unique geodesic that connects $p$ to $q^{(j)}$ in $\overline{\mathcal{U}_j}$. The existence of such a geodesic follows from the simplicity of $(\overline{\mathcal{U}_j}, g_j)$. Recalling that $q^{(j)} \in \mathcal{U}_j$, by extending the geodesic a small distance beyond the point $q^{(j)}$, we obtain a new point $\tilde{q}^{(j)}$ that is a farther distance from $p$ than $q^{(j)}$, reaching a contradiction to the fact that $q^{(j)} \in \mathcal{A}_{M_j, g_j}(p)$. The result follows.
\end{proof}

\begin{appendix}
	
\section{Obstruction to uniqueness in the anisotropic Calder\'on problem for fractional Schr\"odinger equations}

\label{app_obstruction}

The purpose of this appendix is to discuss an obstruction to uniqueness for the inverse problem (IP) stated in the introduction. The following result is presented here for completeness and the convenience of the reader. For similar arguments, see also \cite[Theorem 4.2]{Ghosh_Uhlmann_2021}.
\begin{lemma}
\label{lem_obstruction}
Let $\alpha\in (0,1)$.  Let $(M_j,g_j)$ be a smooth closed Riemannian manifold of dimension $n\ge 2$ and let $V_j\in C^\infty(M_j)$, $j=1,2$. Let $O\subset M_1\cap M_2$ be an open nonempty set such that $M_j\setminus\overline{O}\ne 0$, $j=1,2$. Assume that there is a smooth diffeomorphism $\Phi: M_1\to M_2$ such that $g_1= \Phi^\star g_2$,  $\Phi|_{O}=\text{Id}$, and $V_1= V_2\circ\Phi$. Then 
\begin{equation}
\label{eq_600_0_lemma}
\mathcal{C}_{M_2, g_2, V_2}^O = \mathcal{C}_{M_1, g_1, V_1}^O. 
\end{equation}
\end{lemma}

\begin{proof}
First, since $\Phi$ is a Riemannian isometry, we note that 
\begin{equation}
\label{eq_600_1}
(-\Delta_{g_1})(u \circ \Phi) = (-\Delta_{g_2} u) \circ \Phi,
\end{equation}
for all $u \in C^\infty(M_2)$; see \cite[pages 99, 100]{GPR_book}. Secondly, we claim that the map
\[
U: L^2(M_2) \to L^2(M_1), \quad u \mapsto u \circ \Phi,
\]
is unitary. Indeed, we have  
\[
\|u \circ \Phi\|_{L^2(M_1)}^2 = \int_{M_1} |u \circ \Phi|^2 dV_{g_1} = \int_{M_2} |u|^2 dV_{g_2} = \|u\|_{L^2(M_2)}^2,
\]
see \cite[page 78]{GPR_book}, showing the claim. Hence, rewriting \eqref{eq_600_1} as 
\[
(-\Delta_{g_1}) = U \circ (-\Delta_{g_2}) \circ U^{-1},
\]
and using the functional calculus of self-adjoint operators, we get 
\begin{equation}
\label{eq_600_2}
(-\Delta_{g_1})^\alpha = U \circ (-\Delta_{g_2})^\alpha \circ U^{-1}.
\end{equation}

Now let $u_2 \in C^\infty(M_2)$ be such that $(-\Delta_{g_2})^\alpha u_2 + V_2 u_2 = 0$ on $M_2 \setminus \overline{O}$. This, together with \eqref{eq_600_2}, implies that 
\[
0 = (-\Delta_{g_2})^\alpha u_2 + V_2 u_2 = (-\Delta_{g_1})^\alpha (u_2 \circ \Phi) \circ \Phi^{-1} + V_1 \circ \Phi^{-1} (u_2 \circ \Phi) \circ \Phi^{-1} \quad \text{on} \quad M_2 \setminus \overline{O},
\]
showing $u_1 := u_2 \circ \Phi \in C^\infty(M_1)$ satisfies $(-\Delta_{g_1})^\alpha u_1 + V_1 u_1 = 0$ on $M_1 \setminus \overline{O}$. Note that here we used that $\Phi: M_1 \setminus \overline{O} \to M_2 \setminus \overline{O}$ is a smooth diffeomorphism, thanks to the fact that $\Phi|_{\overline{O}} = \text{Id}$. Furthermore, we have $u_2|_{O} = u_1|_{O}$, and using \eqref{eq_600_2}, we get $((-\Delta_{g_2})^\alpha u_2)|_{O} = ((-\Delta_{g_1})^\alpha u_1)|_{O}$, showing that $\mathcal{C}_{M_2, g_2, V_2}^O \subset \mathcal{C}_{M_1, g_1, V_1}^O$. The opposite inclusion can be proved in a similar way. This establishes \eqref{eq_600_0_lemma}.
\end{proof}

\end{appendix}

\section*{Acknowledgements}

The research of K.K. is partially supported by the National Science Foundation (DMS 2109199 and DMS 2408793). The research of G.U. is partially supported by NSF, a Robert R. and Elaine F. Phelps Endowed Professorship at the University of Washington, and a Si-Yuan Professorship at IAS, HKUST.

\end{document}